\documentclass[12pt,reqno]{amsart}

\usepackage[usenames,dvipsnames]{color}
\usepackage{amssymb}
\usepackage{comment}
\usepackage{enumitem}
\usepackage{graphicx}
\usepackage{esint}
\usepackage{pdfpages}
\usepackage{svg}
\usepackage{caption}
\usepackage{mathrsfs}
\usepackage{xypic}
\usepackage{bm} 
\usepackage[unicode,colorlinks=true,linkcolor=Red,citecolor=Green]{hyperref}

\AtBeginDocument{\addtocontents{toc}{\protect\setlength{\parskip}{0pt}}}

\SetLabelAlign{parright}{\parbox[t]{\labelwidth}{\raggedleft#1}}

\newtheorem{theorem}{Theorem}[section]
\newtheorem{proposition}[theorem]{Proposition}
\newtheorem{lemma}[theorem]{Lemma}
\newtheorem{corollary}[theorem]{Corollary}

\theoremstyle{definition}
\newtheorem{definition}[theorem]{Definition}

\theoremstyle{remark}
\newtheorem{remark}[theorem]{Remark}

\numberwithin{equation}{section}

\def\CIc{C^\infty_{\mathrm c}} 
\def\indic{\operatorname{1\hskip-2.75pt\relax l}} 

\DeclareMathOperator{\bad}{bad}
\DeclareMathOperator{\const}{const}

\newcommand{\fol}{\mathrm{fol}} 

\def\nlin{\mathrm{nlin}} 

\renewcommand{\Re}{\operatorname{Re}}
\newcommand{\res}{\mathrm{res}} 

\DeclareMathOperator{\supp}{supp}

\newcommand{\uni}{\mathrm{uni}} 

\title{Nonlinear uncertainty principle}

\author{Alex Cohen}
\email{alexcohen@nyu.edu}

\address{Courant Institute of Mathematical Sciences, New York University, 251 Mercer Street, New York, NY 10012}

\author{Semyon Dyatlov}
\email{dyatlov@math.mit.edu}

\address{Massachusetts Institute of Technology, 77 Massachusetts Ave, Cambridge, MA 02139}

\begin{document}

\begin{abstract}
We prove an uncertainty principle for subsets of $\mathbb R$ under weak branching assumptions. We impose a nonlinearity condition
on the phase; without it our uncertainty principle fails e.g. for arithmetic progressions.
The key ingredient in the proof is Shmerkin's nonlinear projection theorem.
This paper is a part of the work in progress~\cite{ehalf}, but it is released
earlier than the rest of that paper since it is also used in a paper
of Cohen--Li--Tao~\cite{Cohen-Li-Tao}. It also has an overlap with
our recent paper~\cite{fup12}; see the end of the introduction for more details. 
\end{abstract}

\maketitle

\section{Introduction}

We study operators of the form
\begin{equation}
  \label{e:B-NUP-def}
\mathcal B=\mathcal B(h):L^2(\mathbb R)\to L^2(\mathbb R),\quad
\mathcal Bf(x)=(2\pi h)^{-\frac12}\int_{\mathbb R} e^{{i\over h}\Theta(x,y)}\chi_{\mathcal B}(x,y)f(y)\,dy.
\end{equation}
Here $h>0$ is a small parameter; most of our statements are only interesting in the limit $h\to 0$.
We assume that the phase function $\Theta$ and the amplitude $\chi_{\mathcal B}$ satisfy
\begin{equation}
  \label{e:NUP-Theta-chi}
\Theta\in C^\infty(\mathcal U_{\mathcal B};\mathbb R),\quad
\chi_{\mathcal B}\in \CIc(\mathcal U_{\mathcal B}),\quad
\supp\chi_{\mathcal B}\subset \mathcal K_{\mathcal B}
\end{equation}
where $\mathcal U_{\mathcal B}\subset \mathbb R^2$ is some open set
and $\mathcal  K_{\mathcal B}\subset\mathcal U_{\mathcal B}$ is compact.
We also make a standard \emph{nondegeneracy assumption} on~$\Theta$:
\begin{equation}
  \label{e:NUP-nondegenerate}
\partial_x\partial_y\Theta\neq 0\quad\text{on }\mathcal U_{\mathcal B}.
\end{equation}
If $\Theta,\chi_{\mathcal B}$ are $h$-independent
then the family of operators $\mathcal B(h)$ is a semiclassical Fourier integral operator
associated to the symplectomorphism $\varkappa_\Theta$ of open subsets of $T^*\mathbb R$
defined as follows:
\begin{equation}
  \label{e:kappa-Theta-NUP}
(x,\xi)=\varkappa_\Theta(y,\eta)
\ \iff\
(x,y)\in\mathcal U_{\mathcal B},\
\xi=\partial_x\Theta(x,y),\
\eta=-\partial_y\Theta(x,y).
\end{equation}

An important special case is when $\Theta(x,y)=-xy$ and $\mathcal U_{\mathcal B}=\mathbb R^2$. If we make the illegal choice $\chi_\mathcal B\equiv 1$, then $\mathcal B$ is the semiclassically rescaled unitary Fourier transform,
denoted by~$\mathcal F_h$:
\begin{equation}
\label{e:F-h-def}
\mathcal F_h f(x)=(2\pi h)^{-\frac12}\int_{\mathbb R}e^{-{i\over h}xy}f(y)\,dy.
\end{equation}
In this case the symplectomorphism $\varkappa_\Theta$ is given by
\begin{equation}
  \label{e:F-h-kappa}
\varkappa_\Theta(y,\eta)=(\eta,-y).
\end{equation}
Since $\varkappa_\Theta$ is linear, we call the phase function $\Theta=-xy$ a \emph{linear phase}.

We will study the norm
\begin{equation}
  \label{e:B-NUP-norm}
\|\indic_X\mathcal B\indic_Y\|_{L^2\to L^2}
\end{equation}
where $X,Y\subset\mathbb R$ are some sets (depending on~$h$)
and $\indic_X,\indic_Y:L^2\to L^2$ are the multiplication
operators by the indicator functions of $X,Y$. More precisely, we are interested in \emph{uncertainty principles} which show that this norm goes to~0 as some power of~$h$ as $h\to 0$.

Here is an informal justification for using the term `uncertainty principle':
To get decay of the norm~\eqref{e:B-NUP-norm}, we need to show that for each $f\in L^2(\mathbb R)$,
the function $\mathcal B\indic_Y f$ is small on~$X$. The function $\mathcal B\indic_Y f$
is microlocalized on the set $\varkappa_\Theta(\{y\in Y\})$, and $\indic_X$ microlocalizes
to the set $\{x\in X\}\subset T^*\mathbb R$, so our uncertainty principle
amounts to showing that no function can be microlocalized
on $\varkappa_\Theta(\{y\in Y\})$ and $\{x\in X\}$ at the same time.

\subsection{Nonlinear phase condition and statement of the result}

We will assume the following \emph{nonlinear phase condition}:
\begin{equation}
  \label{e:NUP-curved}
\partial_x\partial_y\log|\partial_x\partial_y\Theta|\neq 0\quad\text{on }\mathcal U_{\mathcal B}.
\end{equation}
This is the key novel assumption of our uncertainty principle.
Note that it fails for the phase function~$\Theta(x,y)=-xy$ and in fact for any phase function $\Theta$
such that the symplectomorphism
$\varkappa_\Theta$ given by~\eqref{e:kappa-Theta-NUP} is linear; this justifies the term `nonlinear phase'.
In the applications to hyperbolic surfaces, such as the one in~\cite{Cohen-Li-Tao}, one uses the phase function $\Theta(x,y)=w\log|x-y|$ (where $w>0$, $x,y$ are on the unit circle,
and $|x-y|$ is the Euclidean distance) and this phase function
satisfies the assumption~\eqref{e:NUP-curved}.
We further discuss the assumption~\eqref{e:NUP-curved} in~\S\ref{s:NUP-discussion} below.

We can now state the main result of this paper. We use the notation introduced in~\S\ref{s:sets-notation} below. In particular, for a set $Y\subset\mathbb R$ and $r>0$,
we denote by $\mathbf M_r(Y)$ the supremum of the Lebesgue measure $|Y\cap I|$ where
$I$ is an $r$-interval.
\begin{theorem}[Nonlinear Uncertainty Principle]
  \label{t:NUP}
For each $\zeta\in (0,0.1]$ there exists $\beta\in (0,\zeta/2]$ such that for all $\epsilon\in (0,0.1\beta]$
and $h\in (0,1]$ the following holds. Assume that $\mathcal B$ is an operator of the form~\eqref{e:B-NUP-def},
satisfying the conditions~\eqref{e:NUP-Theta-chi}, \eqref{e:NUP-nondegenerate}, and~\eqref{e:NUP-curved}.
Assume next that $C_0>0$, $\mathfrak m_Y>0$, and $X,Y\subset \mathbb R$ are two sets satisfying the following conditions:
\begin{align}
  \label{e:NUP-XY-1}
\mathbf M_{|Y(h^{\frac12})|}(Y)
&\leq C_0h^{\zeta}\mathfrak m_Y,
\\
  \label{e:NUP-XY-2}
\mathbf M_{h^{\beta}}(Y)
&\leq C_0h^{33\epsilon}\mathfrak m_Y,
\\
  \label{e:NUP-XY-3}
\mathbf M_r(Y)
&\leq C_0r^{1-\zeta}h^{\zeta-\beta}
\quad\text{for all }r\in [h,h^{\frac12}],
\\
  \label{e:NUP-XY-vol}
|X|\cdot \mathfrak m_Y
&\leq C_0h^{1-\epsilon},\quad
|Y|\leq \mathfrak m_Y.
\end{align}
Then we have the following estimate:
\begin{equation}
  \label{e:NUP}
\|\indic_X\mathcal B\indic_Y\|_{L^2\to L^2}\leq C h^\epsilon.
\end{equation}
Here the constant $C$ depends
only on $C_0$, $\mathcal U_{\mathcal B}$, $\mathcal K_{\mathcal B}$, some
$(\zeta,\epsilon)$-dependent $C^\infty$-seminorms of $\Theta,\chi_{\mathcal B}$,
and upper bounds on $|\partial_x\partial_y\Theta|^{-1}$
and $|\partial_x\partial_y\log|\partial_x\partial_y\Theta||^{-1}$.
\end{theorem}
Theorem~\ref{t:NUP} is reduced in~\S\ref{s:NUP-reductions} to an additive energy
bound, Proposition~\ref{l:NUP-additive}. The latter is proved in~\S\ref{s:NUP-additive-proof}
using a projection theorem of Shmerkin, Theorem~\ref{t:shmerkin-projection}.

We give a few remarks on the objects and assumptions in the theorem:
\begin{itemize}
\item The important special case is when $\mathfrak m_Y=|Y|$; in fact,
later in~\eqref{e:reverse-volume-XY} we essentially reduce to this case.
The advantage of the constant $\mathfrak m_Y$ is monotonicity under inclusion: if some sets $X,Y$ satisfy the conditions~\eqref{e:NUP-XY-1}--\eqref{e:NUP-XY-vol}, then any subsets $X'\subset X,Y'\subset Y$
also satisfy these conditions with the same values of $\zeta,\beta,\epsilon,C_0,h,\mathfrak m_Y$, as follows from~\eqref{e:monotone-set} and~\eqref{e:M-r-monotone}.
\item If the neighborhood $Y(h^{\frac12})$ is an interval,
then $\mathbf M_{|Y(h^{\frac12})|}(Y)=|Y|$. The assumption~\eqref{e:NUP-XY-1} is a quantitative way of ruling out this situation.
\item If $Y$ is contained in an $h^{\beta}$-interval, then $\mathbf M_{h^{\beta}}(Y)=|Y|$.
The assumption~\eqref{e:NUP-XY-2} is a quantitative way of requiring the set $Y$ to be spread out
on scales above~$h^{\beta}$.
\item If $Y$ contains an $r$-interval with $r\gg h^{1-\beta/\zeta}$,
then $\mathbf M_r(Y)=r\gg r^{1-\zeta}h^{\zeta-\beta}$. The assumption~\eqref{e:NUP-XY-3} in particular
rules out this situation.
\item The volume bound~\eqref{e:uncertainty-volume-2} together with~\eqref{e:NUP-XY-vol}
gives the bound $\|\indic_X\mathcal B\indic_Y\|_{L^2\to L^2}\leq Ch^{-\epsilon/2}$, which
is just slightly worse than~\eqref{e:NUP}. That is, to get~\eqref{e:NUP} we only
need a small improvement over the volume bound.
\end{itemize}
In~\S\ref{s:NUP-examples} we give some examples of situations to which Theorem~\ref{t:NUP} applies.
In particular, we discuss its relation to fractal uncertainty principles,
and study the case of arithmetic progressions. The latter (as explained in~\S\ref{s:NUP-discussion})
gives a case where Theorem~\ref{t:NUP} fails for the linear phase $\Theta(x,y)=-xy$.

\subsection{Acknowledgements and disclosures}

The present paper is part of a longer work in progress~\cite{ehalf}. We decided to release it before
the other parts of~\cite{ehalf} are completed, since it is also used in the work
of Cohen--Li--Tao~\cite{Cohen-Li-Tao}.

There is also overlap (both in the ideas used and in the text, specifically in~\S\ref{s:NUP-reductions})
with our recent work on the fractal uncertainty principle, \cite{fup12}. Both papers
use essentially the same reduction to an additive energy quantity via an $L^4$ norm bound.
However, the present paper makes weak branching assumptions on the sets~$X$, $Y$ and obtains a small inexplicit
uncertainty exponent, while \cite{fup12} makes strong Ahlfors--David regularity assumptions on $X$, $Y$
and obtains an explicit regularity exponent (namely, $\frac{1}{64}$ for sets of dimension $\frac12$).
There is some similarity between how the additive energy bound is proved in the present paper
and in~\cite{fup12}, however the details of the proofs are different.
In particular, the present paper uses Shmerkin's nonlinear projection theorem
and~\cite{fup12} uses the recent two-ends Furstenberg theorem of O'Regan--Wu--Yi.

When the work~\cite{ehalf} is finished, we are planning to make the present paper obsolete
(as it will be a part of~\cite{ehalf}) and revise~\cite{fup12} to replace part of it with
references to~\cite{ehalf}. In particular, we do not intend to submit for publication
papers with overlapping text.

Alex Cohen was supported by a Clay Research Fellowship. Semyon Dyatlov was supported by
the NSF grant DMS-2400090, the Simons Fellowship SFI-MPS-SFM-00005584,
and has enjoyed the hospitality of the Mathematical Sciences Research Institute in Berkeley, California, during Fall 2024 (supported by the NSF grant DMS-1928930).

AI was used in a minor way, specifically to find the reference~\cite{Etayo-bilagrangian}.

\section{Uniform trees and sets}

A subset $X\subset\mathbb R$ is called \emph{uniform} if
for any two intervals $J,J'$ of the same length centered at points in~$X$, we have
$|X\cap J|\sim |X\cap J'|$.
(See Definition~\ref{d:set-uniformity} below for details.)
We think of a uniform set as having at each scale a well-defined dimension
which does not change from point to point. However, the dimension can depend on the scale,
so uniformity is a weaker property than Ahlfors--David regularity of dimension $\delta\in [0,1]$:
the latter requires that $|X\cap J|\sim |X|\cdot |J|^\delta$.
In this section we
describe a decomposition of a set $X$ into a union of uniform sets
and a small remainder. This follows the papers of Bourgain~\cite[\S2]{Bourgain-Sum-Product}
and Keleti--Shmerkin~\cite[\S3.1]{Keleti-Shmerkin}. 

We start with the uniform decomposition in the setting of trees (\S\ref{s:uniform-trees})
and later reduce sets to trees by a discretization procedure (\S\ref{s:uniform-sets}).

\subsection{Uniform trees}
\label{s:uniform-trees}

We first discuss the notion of uniformity of trees and show how to split a uniform tree off of
an arbitrary tree (Lemma~\ref{l:single-pigeon}).

\subsubsection{Basic definitions}
\label{s:tree-basics}

Our trees will be associated to sets of words over an alphabet; this provides additional
structure which is not needed in this section but will be present in our applications.
(This is related to the structure called \emph{trie} in computer science.)

Let $\mathcal W$ be a nonempty finite set and $R\geq 0$ be an integer.
Denote by $\mathcal W^R$ the set of words of length $R$ over the alphabet $\mathcal W$:
\begin{equation}
  \label{e:words-def}
\mathcal W^R:=\{\mathbf w=w_0\dots w_{R-1}\mid w_0,\dots,w_{R-1}\in\mathcal W\}.
\end{equation}
For two words $\mathbf w'\in\mathcal W^{T'},\mathbf w''\in\mathcal W^{T''}$, denote
by $\mathbf w'\mathbf w''\in\mathcal W^{T'+T''}$ their concatenation. 

For $0\leq T\leq R$, define the map
\begin{equation}
  \label{e:pi-T-def}
\pi_T:\mathcal W^R\to \mathcal W^T,\quad
\pi_T(w_0\dots w_{R-1})=w_0\dots w_{T-1}.
\end{equation}
If $\mathbf w\in\mathcal W^T$, $\mathbf w'\in\mathcal W^{T'}$ with $T'\leq T$,
then we say
$$
\mathbf w'\prec \mathbf w
$$
if $\mathbf w'$ is a prefix of~$\mathbf w$, that is $\mathbf w'=\pi_{T'}(\mathbf w)$.

A \emph{tree of depth $R$ over the alphabet $\mathcal W$} is defined to be a nonempty subset $\mathbf X\subset\mathcal W^R$.
For such a tree and each integer $T\in [0,R]$, define the set of \emph{nodes
of depth $T$} as the set $\pi_T(\mathbf X)\subset \mathcal W^T$ of all prefixes of length $T$ of the words in $\mathbf X$.
Note that $\pi_0(\mathbf X)$ consists of the empty word and $\pi_R(\mathbf X)=\mathbf X$.
We can draw the tree $\mathbf X$ by placing all the nodes in $\pi_T(\mathbf X)$
in a line at depth $T$ and connecting each $w_0\dots w_{T-1}\in\pi_T(\mathbf X)$,
$T\geq 1$, to its parent $w_0\dots w_{T-2}$.

For $\mathbf w'\in\pi_T(\mathbf X)$, define the \emph{subtree rooted at $\mathbf w'$}
as the set of all elements of $\mathbf X$ having $\mathbf w'$ as a prefix:
\begin{equation}
  \label{e:subtree-rooted}
\mathbf X_{\succ\mathbf w'}:=\pi_T^{-1}(\mathbf w')\cap \mathbf X=\{\mathbf w\in\mathbf X\mid \mathbf w'\prec\mathbf w\}.
\end{equation}
For later use, we introduce the following notation for the size of the largest subtree rooted at a node of depth $T\in [0,R]$:
\begin{equation}
  \label{e:M-T-tree-def}
\mathbf M_T(\mathbf X):=\max_{\mathbf w'\in\pi_T(\mathbf X)}\#(\mathbf X_{\succ\mathbf w'}).
\end{equation}
If $\widetilde{\mathbf Y}\subset\pi_T(\mathbf X)$ is a nonempty subset, we define the \emph{pruning of $\mathbf X$ at depth $T$ by the set~$\widetilde{\mathbf Y}$} as the tree
$$
\pi_T^{-1}(\widetilde{\mathbf Y})\cap \mathbf X=\bigsqcup_{\mathbf w'\in\widetilde{\mathbf Y}}\mathbf X_{\succ\mathbf w'}.
$$
We define a \emph{tree measure} (of depth $R$) as a pair $(\mathbf X,\mu)$
where $\mathbf X$ is a tree of depth $R$ and $\mu:\mathbf X\to [0,\infty)$
is not identically~0. For a subset $\mathbf Y\subset\mathbf X$, define
\begin{equation}
  \label{e:mu-of-a-set}
\mu(\mathbf Y):=\sum_{\mathbf w\in\mathbf Y}\mu(\mathbf w).
\end{equation}

\subsubsection{Uniformity of trees}

Let $0<r\leq 1$ and $0\leq T\leq R$. We say that a tree measure $(\mathbf X,\mu)$ is \emph{$r$-uniform at depth $T$}, if the measures of subtrees rooted at depth $T$ are comparable to each other,
more precisely
\begin{equation}
\label{e:uniform-def}
\min_{\mathbf w'\in \pi_T(\mathbf X)}\mu(\mathbf X_{\succ\mathbf w'})
\geq r\max_{\mathbf w'\in \pi_T(\mathbf X)}\mu(\mathbf X_{\succ\mathbf w'}).
\end{equation}
Note that this is automatically true when $T=0$.
We say $(\mathbf X,\mu)$ is \emph{completely $r$-uniform} if it is $r$-uniform
at every depth.

The next lemma shows that if a tree measure is $r$-uniform at some depth $T'$,
then pruning it at any smaller depth $T\leq T'$ does not destroy this property:
indeed, this pruning completely keeps some subtrees rooted at depth $T'$
and completely eliminates the others.
\begin{lemma}
\label{l:still-uniform}
Assume that a tree measure $(\mathbf X,\mu)$ is $r$-uniform at depth $T'$.
Then any pruning of $\mathbf X$ at depth $T\in [0,T']$
is also $r$-uniform at depth $T'$ with respect to $\mu$.
\end{lemma}
\begin{proof}
Let $\widetilde{\mathbf Y}\subset\pi_T(\mathbf X)$ be a nonempty set and define the pruning
$$
\mathbf Y:=\pi_T^{-1}(\widetilde{\mathbf Y})\cap\mathbf X.
$$
We need to prove that $(\mathbf Y,\mu)$ is $r$-uniform at depth $T'$.
For each $\mathbf w'\in \pi_{T'}(\mathbf Y)$
we have equality of subtrees
\begin{equation}
  \label{e:stilbra-1}
\mathbf Y_{\succ\mathbf w'}=\mathbf X_{\succ\mathbf w'}.
\end{equation}
Indeed, let $\mathbf w''\in\mathcal W^T$ be the prefix of $\mathbf w'$ of length~$T$.
Then $\mathbf X_{\succ\mathbf w'}\subset\pi_T^{-1}(\mathbf w'')$
and $\mathbf Y_{\succ\mathbf w'}=\mathbf X_{\succ\mathbf w'}\cap \pi_T^{-1}(\widetilde{\mathbf Y})$.
We have two cases:
\begin{itemize}
\item either $\mathbf w''\in\widetilde{\mathbf Y}$, then $\mathbf Y_{\succ\mathbf w'}=\mathbf X_{\succ\mathbf w'}$ as needed;
\item of $\mathbf w''\notin\widetilde{\mathbf Y}$, then $\mathbf Y_{\succ\mathbf w'}=\emptyset$,
contradicting the assumption that $\mathbf w'\in\pi_{T'}(\mathbf Y)$.
\end{itemize}
Now \eqref{e:stilbra-1} implies that the set $\{\mu(\mathbf Y_{\succ\mathbf w'})\mid \mathbf w'\in \pi_{T'}(\mathbf Y)\}$
is contained in the set $\{\mu(\mathbf X_{\succ\mathbf w'})\mid \mathbf w'\in\pi_{T'}(\mathbf X)\}$,
so if~\eqref{e:uniform-def} holds for $\mathbf X$ it also holds for $\mathbf Y$.
\end{proof}

\subsubsection{Splitting off a uniform tree}
\label{s:branching-dec-1}

We now show that for each tree measure we can choose a subtree of measure bounded below
which is $r$-uniform. We start with a basic pigeonholing lemma:
\begin{lemma}
  \label{l:basic-pigeon}
Let $S$ be a finite set and $\mu:S\to [0,\infty)$ be a function which
is not identically~0. Fix $r\in (0,\tfrac 12]$. Then there exists a set $Q\subset S$ such that,
using the notation~\eqref{e:mu-of-a-set},
\begin{itemize}
\item $\min_{w\in Q}\mu(w)\geq r\max_{w\in Q}\mu(w)$, and
\item $\mu(Q)\geq \big(6+2{\log\#(S)\over\log(1/r)}\big)^{-1}\mu(S)$.
\end{itemize}
\end{lemma}
\begin{proof}
Dividing $\mu$ by $\mu(S)$, we may assume that $\mu(S)=1$.
For an integer $n\geq 0$, define the set
$$
Q_n:=\{w\in S\mid r^{n+1}<\mu(w)\leq r^n\}.
$$
Then $\bigsqcup_{n\geq 0}Q_n=\{w\in S\mid \mu(w)>0\}$ and thus
$$
1=\sum_{n\geq 0}\mu(Q_n).
$$
Since each $Q_n$ has no more than $\#(S)$ elements, we have
$\mu(Q_n)\leq r^n \#(S)$. Put
$$
n_0:=\bigg\lceil{\log(4\#(S))\over \log (1/r)}\bigg\rceil\leq 3+{\log\#(S)\over \log(1/r)}.
$$
Then
$$
\sum_{n\geq n_0}\mu(Q_n)\leq \#(S)\sum_{n\geq n_0}r^n\leq 2\#(S)r^{n_0}\leq\tfrac 12.
$$
Therefore $\sum_{0\leq n<n_0}\mu(Q_n)\geq \frac 12$, and there exists
$n$ such that $\mu(Q_n)\geq \frac 1{2n_0}$. It remains to put $Q:=Q_n$
for such a choice of $n$.
\end{proof}
We can now give the main result of this subsection:

\begin{figure}
\includegraphics[scale=0.75]{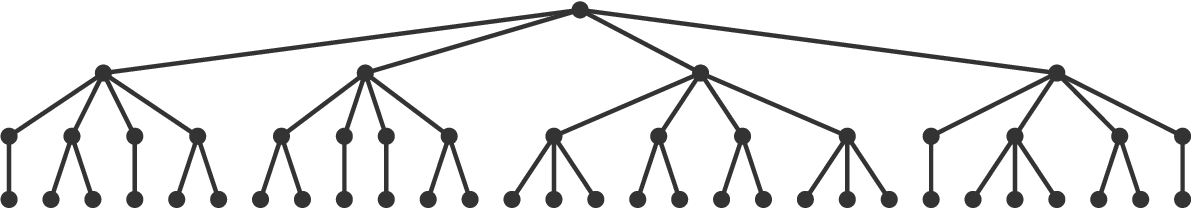}
\vskip.2in
\includegraphics[scale=0.75]{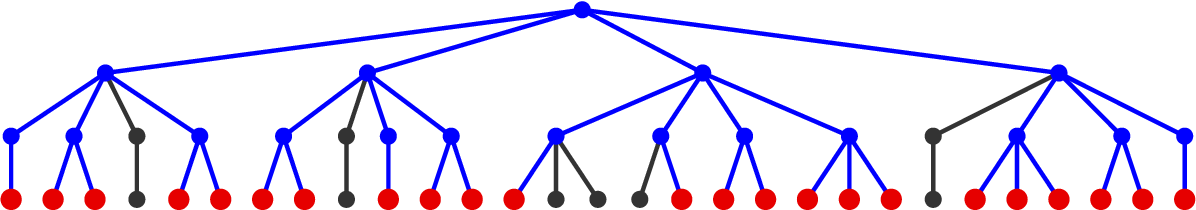}
\vskip.2in
\includegraphics[scale=0.75]{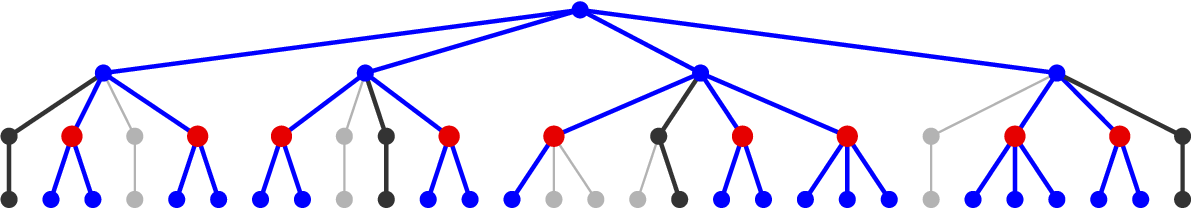}
\vskip.2in
\includegraphics[scale=0.75]{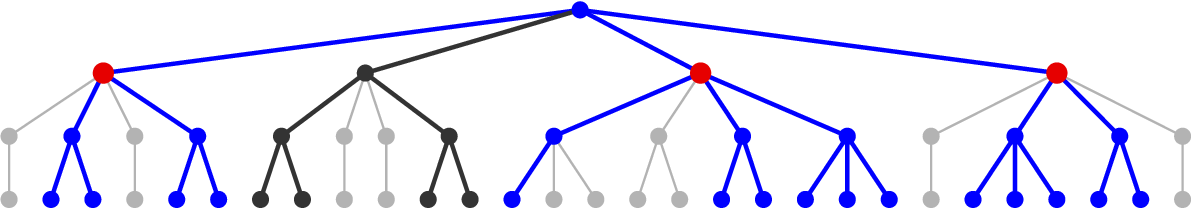}
\vskip.2in
\includegraphics[scale=0.75]{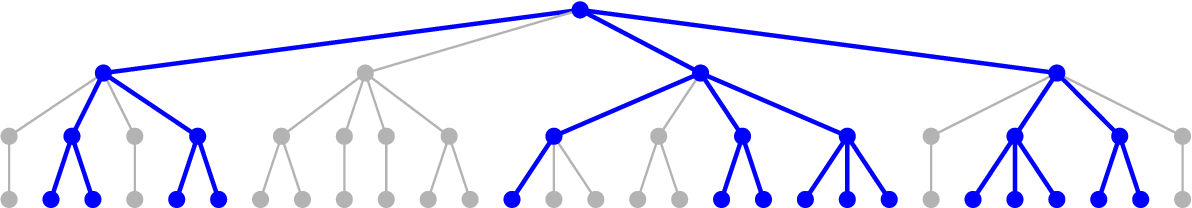}
\caption{An example of the construction of the trees
$\mathbf Y^T$ in Lemma~\ref{l:single-pigeon}, where we put $R:=3$.
The top is the original tree $\mathbf X$. The next picture shows the tree $\mathbf Y^T$ for $T=3$
(red/blue).
The next 2 pictures
show the trees $\mathbf Y^{T+1}$ (in black) for $T=2,1$, the sets $\widetilde{\mathbf Y}^T$ (in red), and the trees $\mathbf Y^T$ (in blue).
At the bottom is the final tree $\mathbf Y=\mathbf Y^1$. At each step $T=R,\dots,1$,
we prune the tree $\mathbf Y^{T+1}$ at depth $T$ by
a set constructed in Lemma~\ref{l:basic-pigeon} to get a tree $\mathbf Y^T$ which is $r$-uniform
at depth~$T$, while also keeping it $r$-uniform at larger depths by Lemma~\ref{l:still-uniform}
and retaining at least $\big(6+2R{\log \#(\mathcal W)\over\log(1/r)}\big)^{-1}$ fraction of the measure $\mu$.}
\label{f:tree-pruning}
\end{figure}

\begin{lemma}[Splitting off a uniform tree]
\label{l:single-pigeon}
Let $(\mathbf X,\mu)$ be a tree measure (with $\mathbf X\subset \mathcal W^R$) and let $r\in (0,\frac12]$. Then there exists a subset $\mathbf Y\subset \mathbf X$ such that $(\mathbf Y,\mu)$ is completely $r$-uniform and
\begin{equation}
  \label{e:single-pigeon}
\mu(\mathbf Y)\geq \Big(6+2R{\log \#(\mathcal W)\over\log(1/r)}\Big)^{-R}\mu(\mathbf X).
\end{equation}
\end{lemma}
\begin{proof}
We prune the tree $\mathbf X$ from the bottom to the top. More precisely,
we define the nested collection of subsets of $\mathbf X$
$$
\mathbf Y^R\supset \mathbf Y^{R-1}\supset\dots\supset \mathbf Y^1
$$
as follows (see also Figure~\ref{f:tree-pruning}):
\begin{itemize}
\item Applying Lemma~\ref{l:basic-pigeon} to the map $\mu:\mathbf X\to [0,\infty)$, we construct a set $\mathbf Y^R\subset\mathbf X$ such that (here we use
that $\#(\mathbf X)\leq \#(\mathcal W)^R$)
\begin{align}
  \label{e:singpig-1}
\min_{\mathbf w\in \mathbf Y^R}\mu(\mathbf w)&\geq r\max_{\mathbf w\in \mathbf Y^R}\mu(\mathbf w),
\\
  \label{e:singpig-2}
\mu(\mathbf Y^R)&\geq \Big(6+2R{\log\#(\mathcal W)\over\log(1/r)}\Big)^{-1}\mu(\mathbf X).
\end{align}
\item Given the set $\mathbf Y^{T+1}$ for some $T\in [1,R-1]$, we construct the set
$\mathbf Y^T$ as follows. Let $\pi_T(\mathbf Y^{T+1})$ be the set of nodes of depth~$T$
in the tree $\mathbf Y^{T+1}$. 
Define the measure
$$
\mu_T:\pi_T(\mathbf Y^{T+1})\to [0,\infty),\quad
\mu_T(\mathbf w')=\mu(\mathbf Y^{T+1}_{\succ\mathbf w'}).
$$
Applying Lemma~\ref{l:basic-pigeon} to the measure $\mu_T$, we construct a set
$\widetilde{\mathbf Y}^T\subset \pi_T(\mathbf Y^{T+1})$ such that
\begin{align}
  \label{e:singpig-1-again}
\min_{\mathbf w'\in \widetilde{\mathbf Y}^T}\mu_T(\mathbf w')
&\geq r\max_{\mathbf w'\in \widetilde{\mathbf Y}^T}\mu_T(\mathbf w'),\\
  \label{e:singpig-2-again}
\mu_T(\widetilde{\mathbf Y}^T)&\geq \Big(6+2R{\log\#(\mathcal W)\over\log(1/r)}\Big)^{-1}\mu_T(\pi_T(\mathbf Y^{T+1})).
\end{align}
Now, define the $\mathbf Y^T$ as the pruning of the tree $\mathbf Y^{T+1}$ at depth
$T$ by the set~$\widetilde{\mathbf Y}^T$:
\begin{equation}
  \label{e:singpig-3}
\mathbf Y^T:=\pi_T^{-1}(\widetilde{\mathbf Y}^T)\cap \mathbf Y^{T+1}.
\end{equation}
Note that $\mu_T(\pi_T(\mathbf Y^{T+1}))=\mu(\mathbf Y^{T+1})$
and $\mu_T(\widetilde{\mathbf Y}^T)=\mu(\mathbf Y^T)$, so~\eqref{e:singpig-2-again} gives
\begin{equation}
\label{e:singpig-2.5-again}
\mu(\mathbf Y^T)\geq \Big(6+2R{\log\#(\mathcal W)\over\log(1/r)}\Big)^{-1}\mu(\mathbf Y^{T+1}).
\end{equation}
Moreover, $\pi_T(\mathbf Y^T)=\widetilde{\mathbf Y}^T$ and for each
$\mathbf w'\in \widetilde{\mathbf Y}^T$, we have
$\mathbf Y^T_{\succ\mathbf w'}=\mathbf Y^{T+1}_{\succ\mathbf w'}$.
\end{itemize}
The tree measure $(\mathbf Y^T,\mu)$ is $r$-uniform at depths
from $T$ to $R$. Indeed, it is $r$-uniform at depth
$T$ by~\eqref{e:singpig-1} and~\eqref{e:singpig-1-again},
and it is $r$-uniform at depths from $T+1$ to $R$ by induction on $T$
and Lemma~\ref{l:still-uniform}.

Now, put $\mathbf Y:=\mathbf Y^1$. Then $(\mathbf Y,\mu)$
is completely $r$-uniform. Moreover, by~\eqref{e:singpig-2} and~\eqref{e:singpig-2.5-again},
we get~\eqref{e:single-pigeon}.
\end{proof}

\subsection{Uniform sets of reals}
\label{s:uniform-sets}

We now introduce the notion of a uniform set and show
how to decompose a general set into a union of uniform sets and a small remainder
(Lemma~\ref{l:uniform-set-decomposition}).

\subsubsection{Notation and basic properties}
\label{s:sets-notation}

We first introduce some basic concepts:
\begin{itemize}
\item We use the sum set notation for $X,Y\subset\mathbb R^d$:
$$
X+Y:=\{x+y\mid x\in X,\ y\in Y\}.
$$
\item For $r>0$, an \emph{$r$-cube} in $\mathbb R^d$ is a translate of $[0,r]^d$.
If $d=1$, we call an $r$-cube an \emph{$r$-interval}.
\item All the subsets of $\mathbb R^d$ we use will be Borel measurable.
\item For $X\subset\mathbb R^d$, denote by $|X|$ its Lebesgue measure.
\item For $X\subset\mathbb R^d$ and $r>0$, define the quantities
\begin{equation}
  \label{e:M-r-X-def}
|X|_r \in [0,\infty],\quad
\mathbf M_r(X)\in [0,r^d]
\end{equation}
as follows: $|X|_r$ is the minimal number of $r$-cubes needed to cover $X$,
and $\mathbf M_r(X)$ is the supremum of $|X\cap Q|$ over all $r$-cubes $Q\subset\mathbb R^d$.
Define also the $r$-neighborhood
\begin{equation}
  \label{e:X-r-def}
X(r):=X+[-r,r]^d.
\end{equation}
\end{itemize}
We record here some basic properties of the above objects:
\begin{itemize}
\item If $X\subset Y\subset\mathbb R^d$ and $r>0$ then
\begin{equation}
  \label{e:monotone-set}
|X|_r\leq |Y|_r,\quad
\mathbf M_r(X)\leq \mathbf M_r(Y).
\end{equation}
\item If $0<r_1\leq r_2$ and $X\subset \mathbb R^d$ then
\begin{align}
  \label{e:r-cover-monotone}
|X|_{r_2}
&\leq |X|_{r_1}\leq \Big\lceil{r_2\over r_1}\Big\rceil^d |X|_{r_2},
\\
  \label{e:M-r-monotone}
\mathbf M_{r_1}(X)
&\leq \mathbf M_{r_2}(X)\leq \Big\lceil{r_2\over r_1}\Big\rceil^d\mathbf M_{r_1}(X).
\end{align}
Indeed, each $r_1$-cube is contained in some $r_2$-cube, and each $r_2$-cube can be covered by $\lceil r_2/r_1\rceil^d$ many $r_1$-cubes.
\item For each $X\subset\mathbb R^d$ and $r>0$ we have
\begin{align}
  \label{e:volume-covering-M}
|X|&\leq |X|_r\,\mathbf M_r(X),
\\
  \label{e:volume-nbhd}
 (r/2)^d|X|_r&\leq |X(r)|\leq (3r)^d|X|_r.
\end{align}
To see the first bound in~\eqref{e:volume-nbhd}, let
$x_1,\dots,x_L\in X$ be a maximal set such that
the $r/2$-cubes centered at $x_j$ do not intersect. These
cubes are contained in~$X(r)$, so comparing the volume we see that
$L(r/2)^d\leq |X(r)|$.
On the other hand, $X$ is covered by the $r$-cubes centered at $x_j$,
so $|X|_r\leq L$.
 
To see the second bound in~\eqref{e:volume-nbhd}, note that
$X$ can be covered by $|X|_r$ many $r$-cubes, so $X(r)$
can be covered by $|X|_r$ many $3r$-cubes.
\end{itemize}

\subsubsection{Uniformity of sets}
  \label{s:set-uniformity}

We can now introduce the notion of a uniform set:
\begin{definition}
  \label{d:set-uniformity}
Let $X\subset \mathbb R$ be a set and $r_1,h\in (0,1]$. We say that $X$ is
\emph{$r_1$-uniform on scales $h$ to~1}, if for each interval $J\subset \mathbb R$
of length $|J|\in [h,1]$
whose center lies in~$X$ we have
\begin{equation}
  \label{e:set-uniformity}
|X\cap J|\geq r_1\mathbf M_{|J|}(X).
\end{equation}
\end{definition}
Most statements in this section are valid for any $r_1,h\in (0,1]$,
but only interesting when $h$ is small and $r_1$ is not too small.

Here are a few examples:
\begin{enumerate}
\item Let $0\leq\alpha\leq 1$ and take $X:=[0,h^\alpha]$ to be an interval.
Then $X$ is $0.01$-uniform on scales $h$ to~1 and we have for all $\tau\in [0,1]$
$$
|X|_{h^\tau}\sim h^{\min(\alpha-\tau,0)},\quad
\mathbf M_{h^\tau}(X)=h^{\max(\tau,\alpha)}.
$$
\item Let $0\leq\alpha\leq 1$ and take $X:=[0,1]\cap (h^\alpha \mathbb Z+[0,h])$
be a union of $h$-intervals arranged in an arithmetic progression of difference~$h^\alpha$.
Then $X$ is $0.01$-uniform on scales $h$ to~1 and we have for all $\tau\in [0,1]$
$$
|X|_{h^\tau}\sim h^{-\min(\tau,\alpha)},\quad
\mathbf M_{h^\tau}(X)\sim h^{1+\min(\tau-\alpha,0)}.
$$
\item Assume that $\delta\in [0,1]$ and $X\subset [0,1]$ is a $\delta$-dimensional Ahlfors--David regular 
set on scales $h$ to~1 with regularity constant $C_R\geq 1$ in the following sense: for each interval $J\subset\mathbb R$
of length $|J|\in [h,1]$ whose center lies in~$X$ we have
\begin{equation}
  \label{e:AD-regular-def}
C_R^{-1} h^{1-\delta}|J|^\delta\leq |X\cap J|\leq C_R h^{1-\delta}|J|^\delta.
\end{equation}
(Other definitions are possible; the present one is tailored to the case when $X$ is a union
of $h$-intervals.)
Then $X$ is $(2C_R^2)^{-1}$-uniform on scales $h$ to~1, and we have for all
$\tau\in [0,1]$
$$
|X|_{h^\tau}\sim h^{-\delta\tau},\quad
\mathbf M_{h^\tau}(X)\sim h^{1-\delta+\delta\tau}.
$$ 
\item The set $X:=[0,h]\sqcup [\tfrac12,1]$ is \textbf{not} $0.01$-uniform on scales~$h$ to~1,
when $h$ is small enough. Indeed, $\mathbf M_{h^{\frac12}}(X)=h^{\frac12}$,
but for $J=[-\frac12h^{\frac12},\frac12h^{\frac12}]$ we have
$|X\cap J|=h$.
\end{enumerate}

\subsubsection{Properties of uniform sets}

We now establish a few properties of uniform sets. These are used in~\S\ref{s:NUP-reduction-proof}.
\begin{lemma}
  \label{l:uniformity-volume}
Assume that $X\subset\mathbb R$ is $r_1$-uniform on scales $h$ to~1, $\tilde r\in [h,1]$, and $J\subset\mathbb R$
is an interval. Then
\begin{equation}
  \label{e:uniformity-volume-refined}
\mathbf M_{|J|+\tilde r}(X)\geq\tfrac12 r_1|X\cap J|_{\tilde r}\,\mathbf M_{\tilde r}(X).
\end{equation}
\end{lemma}
\begin{proof}
Let $x_1,\dots,x_L\in X\cap J$ be a maximal $\tilde r$-separated set
and $J_\ell$ be the $\tilde r$-interval centered at $x_\ell$. Then
\begin{equation}
  \label{e:univol-1}
\mathbf M_{|J|+\tilde r}(X)\geq |X\cap J(\tilde r/2)|\geq \sum_{\ell=1}^L |X\cap J_\ell|\geq Lr_1\mathbf M_{\tilde r}(X).
\end{equation}
Here the first inequality follows since $J(\tilde r/2)$
is an interval of length $|J|+\tilde r$. The second inequality holds since the intervals $J_\ell$ do not overlap for
different $\ell$ and are contained
in $J(\tilde r/2)$. The last inequality follows from the uniformity of~$X$.

On the other hand, the $2\tilde r$-intervals centered at $x_1,\dots,x_L$ cover $X\cap J$, so
by~\eqref{e:r-cover-monotone} we have $|X\cap J|_{\tilde r}\leq 2|X\cap J|_{2\tilde r}\leq 2L$. Together with~\eqref{e:univol-1} this gives~\eqref{e:uniformity-volume-refined}.
\end{proof}
\begin{remark}
Taking $J:=[-N,N]$ where $N\to\infty$ in Lemma~\ref{l:uniformity-volume}, we obtain the following counterpart to~\eqref{e:volume-covering-M}:
\begin{equation}
  \label{e:uniformity-volume}
|X|\geq \tfrac12 r_1|X|_{\tilde r}\,\mathbf M_{\tilde r}(X)\quad\text{for all }\tilde r\in [h,1].
\end{equation}
An informal proof of this is as follows: we can cover $X$ by $|X|_{\tilde r}$ many $\tilde r$-intervals,
and we can adjust these to be centered at points in $X$ and have bounded overlap, so the measure of $X$
is bounded below by the number of the intervals times the measure of~$X$ in each, which by uniformity
is comparable to $\mathbf M_{\tilde r}(X)$.

The reader is encouraged to verify that the sets in Examples~(1)--(3) in~\S\ref{s:set-uniformity}
satisfy the estimate~\eqref{e:uniformity-volume}; the set in Example~(4) satisfies this estimate as well
(even though it is not uniform). Here is an example of a non-uniform set (with $r_1$ independent of~$h$) which does \textbf{not} satisfy~\eqref{e:uniformity-volume} with $\tilde r:=h^{\frac12}$. Let us split $[0,1]$ into `big' intervals of size $h^{\frac12}$,
each of which is split into `small' intervals of size~$h$. We choose one big interval and call it `special' and the other big intervals `generic'. To construct~$X$, take the special big interval and add one small interval in each generic big interval. Then $|X|_{h^{\frac12}}\sim h^{-\frac12}$,
$\mathbf M_{h^{\frac12}}(X)=h^{\frac12}$,
and their product is much larger than $|X|\sim h^{\frac12}$.
\end{remark}
%
\begin{lemma}
  \label{l:uniformity-volume-more}
Assume that $X\subset\mathbb R$ is $r_1$-uniform on scales $h$ to~1, $|X|>0$, and $h\leq \tilde r\leq r\leq 1$. Then
\begin{equation}
  \label{e:uniformity-volume-more}
{\mathbf M_r(X(\tilde r))\over |X(\tilde r)|}
\leq
{48\mathbf M_r(X)\over r_1|X|}.
\end{equation}
\end{lemma}
\begin{remark}
It might be helpful to consider the following example when Lemma~\ref{l:uniformity-volume-more}
is essentially sharp. Let us split $[0,1]$ into `big' intervals of size $h^{\frac13}$,
each of which is split into `medium' intervals of size $h^{\frac23}$, each
of which is split into `small' intervals of size~$h$.
We choose one big interval and call it `special' and the other big intervals `generic'.
To construct $X$, pick one small interval inside each medium interval in the special big interval,
and add one medium interval in each generic big interval.
Put $\tilde r:=h^{\frac23}$ and $r:=h^{\frac13}$.
Then $|X|\sim h^{\frac13}$, $\mathbf M_r(X)\sim h^{\frac23}$, $|X(\tilde r)|\sim h^{\frac13}$, and $\mathbf M_r(X(\tilde r))\sim h^{\frac13}$; the set $X$ satisfies the uniformity condition with $r_1\sim h^{\frac13}$.
\end{remark}
\begin{proof}[Proof of Lemma~\ref{l:uniformity-volume-more}]
Let $J\subset\mathbb R$ be an $r$-interval; we need to show that
$$
|X(\tilde r)\cap J|\leq  {48\mathbf M_r(X)|X(\tilde r)|\over r_1 |X|}.
$$
This is proved as follows:
$$
\begin{aligned}
|X(\tilde r)\cap J|
&\leq |(X\cap J(\tilde r))(\tilde r)|
\leq 3\tilde r|X\cap J(\tilde r)|_{\tilde r}
\\
&\leq
{6\tilde r\mathbf M_{r+3\tilde r}(X)\over r_1 \mathbf M_{\tilde r}(X)}
\leq {24\tilde r\mathbf M_r(X)|X|_{\tilde r}\over r_1|X|}
\leq {48\mathbf M_r(X)|X(\tilde r)|\over r_1|X|}.
\end{aligned}
$$
Here the second inequality follows from~\eqref{e:volume-nbhd}. The third inequality
follows from~\eqref{e:uniformity-volume-refined} since $J(\tilde r)$ is an interval of length $r+2\tilde r$.
The fourth inequality follows from~\eqref{e:M-r-monotone} and~\eqref{e:volume-covering-M}.
The last inequality follows from~\eqref{e:volume-nbhd}.
\end{proof}

\subsubsection{Splitting off a uniform set}

The following lemma shows how to extract a large uniform subset of a given set.
It is proved by discretizing the interval $[0,1]$ using base~$P$ representation
with $R_1$ many digits and applying Lemma~\ref{l:single-pigeon}.
\begin{lemma}
  \label{l:uniform-set-split}
Assume that $X\subset [0,1]$ satisfies $|X|>0$ and $P\geq 2,R_1\geq 1$ are integers. Then
there exists
a subset $Y\subset X$ such that $Y$ is ${1\over P(2P+1)}$-uniform
on scales~$2P^{-R_1}$ to~1 and
\begin{equation}
  \label{e:uniform-set-split}
|Y|\geq (2R_1+6)^{-R_1}|X|.
\end{equation}
\end{lemma}
\begin{proof}
We use the notation for trees, tree measures, and their uniformity from~\S\ref{s:uniform-trees}.

\noindent 1.
Consider the alphabet $\mathcal W:=\{0,1,\dots,P-1\}$.
To each word $\mathbf w=w_0\dots w_{T-1}\in\mathcal W^T$, $T\geq 0$,
we associate the interval
$$
J_{\mathbf w}:=\bigg(\sum_{j=0}^{T-1}w_j P^{-j-1}\bigg)+[0,P^{-T}].
$$
The intervals $(J_{\mathbf w})_{\mathbf w\in\mathcal W^T}$ form a nonoverlapping cover of~$[0,1]$,
and $J_{\mathbf w}\subset J_{\mathbf w'}$ whenever $\mathbf w'\prec\mathbf w$.

Consider the tree measure $(\mathcal W^{R_1},\mu)$ where
$$
\mu(\mathbf w):=|X\cap J_{\mathbf w}|,\quad
\mathbf w\in\mathcal W^{R_1}.
$$
By Lemma~\ref{l:single-pigeon} with $r:=P^{-1}$, there exists
a subset $\mathbf Y\subset \mathcal W^{R_1}$ such that
the tree measure $(\mathbf Y,\mu)$ is completely $P^{-1}$-uniform and
$\mu(\mathbf Y)\geq (2R_1+6)^{-R_1}\mu(\mathcal W^{R_1})$. We now put
$$
Y:=\bigcup_{\mathbf w\in \mathbf Y} (X\cap J_{\mathbf w}).
$$
Then $|Y|=\mu(\mathbf Y)$ and $|X|=\mu(\mathcal W^{R_1})$ which gives~\eqref{e:uniform-set-split}.

Let $T\in \{0,\dots,R_1\}$. Then
$$
|Y\cap J_{\mathbf w'}|=\mu(\mathbf Y_{\succ\mathbf w'})\quad\text{for all }\mathbf w'\in\pi_T(\mathbf Y).
$$
Note also that $|Y\cap J_{\mathbf w'}|=0$ for all $\mathbf w'\in\mathcal W^T\setminus \pi_T(\mathbf Y)$. Denote
$$
c_T:=\max_{\mathbf w'\in \pi_T(\mathbf Y)}\mu(\mathbf Y_{\succ\mathbf w'}).
$$
The uniformity of $(\mathbf Y,\mu)$ implies that
$$
\begin{aligned}
|Y\cap J_{\mathbf w'}|
&\leq c_T\quad\text{for all }\mathbf w'\in\mathcal W^T,
\\
|Y\cap J_{\mathbf w'}|
&\geq P^{-1}c_T\quad\text{for all }\mathbf w'\in\pi_T(\mathbf Y).
\end{aligned}
$$

\noindent 2. We now show the uniformity of the set $Y$.
Take arbitrary $r\in [2P^{-R_1},1]$.
Take the smallest integer $T$ such that $P^{-T}\leq r/2$; note that $1\leq T\leq R_1$.

Let $J$ be any interval of length~$r$. Since $|J|\leq 2P^{1-T}$,
we can cover $J\cap [0,1]$ by at most $2P+1$ intervals $J_{\mathbf w'}$
where $\mathbf w'\in \mathcal W^T$. Each of those satisfies
$|Y\cap J_{\mathbf w'}|\leq c_T$.
It follows that $|Y\cap J|\leq (2P+1)c_T$ and thus
\begin{equation}
  \label{e:ussp-upper}
\mathbf M_r(Y)\leq (2P+1)c_T.
\end{equation}
Now, assume additionally that the center $x$ of the interval $J$ lies in~$Y$.
Then $x\in J_{\mathbf w}$ for some $\mathbf w\in \mathbf Y$. Denote $\mathbf w':=\pi_T(\mathbf w)\in \pi_T(\mathbf Y)$.
Then $x$ lies in the interval $J_{\mathbf w'}$, which has length $P^{-T}\leq |J|/2$.
Since the interval $J$ is centered at $x$, we have $J_{\mathbf w'}\subset J$ and thus
$$
|Y\cap J|\geq |Y\cap J_{\mathbf w'}|\geq P^{-1}c_T.
$$
Together with~\eqref{e:ussp-upper} this shows that
$$
|Y\cap J|\geq {\mathbf M_r(Y)\over P(2P+1)}
$$
and finishes the proof of uniformity of~$Y$.
\end{proof}

\subsubsection{Decomposing into uniform sets}

We now apply Lemma~\ref{l:uniform-set-split} iteratively
to decompose a set into $\mathcal O(h^{-\varepsilon})$ many subsets which are $1/\mathcal O(h^{-\varepsilon})$-uniform on scales~$h$ to~1, for any $\varepsilon>0$:
\begin{lemma}
  \label{l:uniform-set-decomposition}
There exist functions $N_1(h)\geq 1,r_1(h)\in (0,\frac12]$ of $h\in (0,0.1]$ satisfying
\begin{equation}
  \label{e:usd-Nr}
N_1(h),r_1(h)^{-1}=\mathcal O_\varepsilon(h^{-\varepsilon})\quad\text{as }h\to 0\quad\text{for all }\varepsilon>0
\end{equation}
such that for each set $X\subset [0,1]$ with $|X|>0$ and each $h\in (0,0.1]$ we can decompose
\begin{equation}
  \label{e:usd-X}
X=\bigg(\bigsqcup_{j=1}^{N_X} X_j\bigg)\sqcup X_{\res},\quad
N_X\leq N_1(h),\quad
|X_{\res}|\leq h|X|,
\end{equation}
where each set $X_j$ is $r_1(h)$-uniform on scales~$h$ to~1.
\end{lemma}
\begin{proof}
Put
$$
R_1(h):=\big\lfloor\sqrt{\log(1/h)}\big\rfloor
$$
and let $P(h)\geq 2$ be the smallest positive integer such that
$$
2P(h)^{-R_1(h)}\leq h.
$$
Define
$$
r_1(h):={1\over P(h)(2P(h)+1)}.
$$
Since $R_1(h)\to \infty$ as $h\to 0$, we have $r_1(h)^{-1}=\mathcal O_\varepsilon(h^{-\varepsilon})$
for all $\varepsilon>0$.

We construct the sets $(X_j)_{1\leq j\leq N_X}$ and $(Z_j)_{0\leq j\leq N_X}$ iteratively
as follows:
\begin{itemize}
\item We start with $Z_0:=X$ and then consider $j=1,2,\dots$
\item If $|Z_{j-1}|\leq h|X|$, then we stop the process and set $N_X:=j-1$.
\item Otherwise we use Lemma~\ref{l:uniform-set-split} to construct
a subset $X_j\subset Z_{j-1}$ such that $X_j$ is $r_1(h)$-uniform on scales~$h$ to~1 and
\begin{equation}
  \label{e:usd-int-1}
|X_j|\geq (2R_1(h)+6)^{-R_1(h)}|Z_{j-1}|.
\end{equation}
We then put $Z_j:=Z_{j-1}\setminus X_j$.
\end{itemize}
Iterating~\eqref{e:usd-int-1} we see that
$$
|Z_j|\leq \big(1-(2R_1(h)+6)^{-R_1(h)}\big)^j |X|.
$$
Thus the process above terminates and we have $N_X\leq N_1(h)$, where $N_1(h)$ is the smallest integer such that
$$
\big(1-(2R_1(h)+6)^{-R_1(h)}\big)^{N_1(h)}\leq h.
$$
We have
$$
\begin{aligned}
N_1(h)&=\Big\lceil -{\log(1/h)\over \log(1-(2R_1(h)+6)^{-R_1(h)})}\Big\rceil
\\
&=\mathcal O\big((3R_1(h))^{R_1(h)}\log(1/h)\big)
\\
&=\mathcal O(e^{\sqrt{\log(1/h)}\log\log(1/h)})=\mathcal O_{\varepsilon}(h^{-\varepsilon}).
\end{aligned}
$$
Finally, from the construction above we have the partition property~\eqref{e:usd-X}, with the residual set
$X_{\res}=Z_{N_X}$ satisfying $|X_{\res}|\leq h|X|$.
\end{proof}

\section{Nonlinear uncertainty principle}

In this section, we prove Theorem~\ref{t:NUP}. The proof uses the nonlinear projection theorem of Shmerkin~\cite{Shmerkin-nonlinear}, which in turn relies on Bourgain's projection theorem~\cite{Bourgain-Sum-Product}.

\subsection{Basic properties}

We first give
background for uncertainty principles and discuss in detail the assumptions that we make.
Throughout this section we use the spaces $L^p=L^p(\mathbb R^d)$ with respect to the Lebesgue measure.

\subsubsection{Basics of uncertainty principles}
\label{s:uncertainty-basics}

We first give some basic upper bounds on the uncertainty norm~\eqref{e:B-NUP-norm}. We have the \emph{unitarity bound}
\begin{equation}
  \label{e:B-NUP-L2}
\|\indic_X\mathcal B\indic_Y\|_{L^2\to L^2}\leq \|\mathcal B\|_{L^2\to L^2}\leq C
\end{equation}
where the constant $C$ depends only on $\mathcal U_{\mathcal B},\mathcal K_{\mathcal B}$,
some $C^\infty$-seminorms of $\Theta,\chi_{\mathcal B}$, and the upper bound
on $|\partial_x\partial_y\Theta|^{-1}$. See Remark~\ref{l:L2-intermediate-remark} below for the proof.

We also have
\begin{equation}
  \label{e:L1-Linfty}
\|\mathcal B\|_{L^1\to L^\infty}\leq Ch^{-\frac12},\quad
C:=(2\pi)^{-\frac12}\sup|\chi_{\mathcal B}|.
\end{equation}
Together with H\"older's inequality this implies the \emph{volume bound}
\begin{equation}
  \label{e:uncertainty-volume-2}
\|\indic_X\mathcal B\indic_Y\|_{L^2\to L^2}\leq C h^{-\frac12}\sqrt{|X|\cdot |Y|}.
\end{equation}
We also give some lower bounds on the norm~\eqref{e:B-NUP-norm}. We use the following elementary
\begin{lemma}
  \label{l:NUP-lower}
Assume that $X,Y\subset\mathbb R$, $X\times Y\subset\mathcal U_{\mathcal B}$,
and the functions $\Theta,\chi_{\mathcal B}$ satisfy for all $x\in X,y\in Y$
\begin{align}
\label{e:NUP-lower-ass-1}
\Theta(x,y) & \in 2\pi h(\mathbb Z+[-\tfrac16,\tfrac16]),
\\
\label{e:NUP-lower-ass-2}
\chi_{\mathcal B}(x,y) & =1.
\end{align}
Then
\begin{equation}
  \label{e:NUP-lower}
\|\indic_X\mathcal B\indic_Y\|_{L^2\to L^2}\geq (8\pi h)^{-\frac12}\sqrt{|X|\cdot |Y|}.
\end{equation}
\end{lemma}
\begin{proof}
From~\eqref{e:NUP-lower-ass-1} and~\eqref{e:NUP-lower-ass-2} we see that
$$
\Re\big (e^{{i\over h}\Theta(x,y)}\chi_{\mathcal B}(x,y)\big)\geq \tfrac12
\quad\text{for all }x\in X,\ y\in Y.
$$
Let $f:=\mathbf 1_Y$ be the indicator function of~$Y$. Then for all $x\in X$,
we have $\Re\mathcal Bf(x)\geq (8\pi h)^{-\frac12} |Y|$ and thus
$$
\|\indic_X\mathcal B\indic_Y f\|_{L^2}\geq (8\pi h)^{-\frac12}\sqrt{|X|}\cdot|Y|,\quad
\|f\|_{L^2}=\sqrt{|Y|}
$$
which gives~\eqref{e:NUP-lower}.
\end{proof}
An application of Lemma~\ref{l:NUP-lower} is the following basic counterexample
to uncertainty principle which we call the \emph{brick}. Fix $\alpha\in (0,1)$, assume that $\chi_{\mathcal B}=1$ near $(0,0)$,
$\Theta(x,0)=\Theta(0,y)=0$ for all $x,y$, $\partial_x\partial_y\Theta(0,0)=1$, and put
\begin{equation}
  \label{e:brick-ex-1}
X:=[0,0.1 h^\alpha],\quad
Y:=[0,0.1 h^{1-\alpha}].
\end{equation}
Then we have for all $x\in X$, $y\in Y$ and small enough $h$
$$
\Theta(x,y)=\int_0^x\int_0^y\partial_x\partial_y\Theta(x',y')\,dy'dx'\in [0,2xy]\ \subset\ [0, 0.02h].
$$
Thus Lemma~\ref{l:NUP-lower} applies and gives
\begin{equation}
  \label{e:brick-ex-2}
\|\indic_X\mathcal B_h\indic_Y\|_{L^2\to L^2}\geq 0.01.
\end{equation}
To make this example apply to general $\Theta$, we note that replacing
$\Theta(x,y)$ by $\Theta(x,y)+a_1(x)+a_2(y)$ for any smooth functions $a_1,a_2$ does not change
the norm~\eqref{e:B-NUP-norm} since this changes the operator $\mathcal B$
to $e^{{i\over h}a_1}\mathcal Be^{{i\over h}a_2}$. Thus we can reduce to the case
when $\Theta(x,0)=\Theta(0,y)=0$ for all $x,y$. The restriction $\partial_x\partial_y\Theta(0,0)=1$
is not essential either as can be seen by examining the argument above.

\subsubsection{Examples}
\label{s:NUP-examples}

We next discuss whether the assumptions of Theorem~\ref{t:NUP} apply to the examples~(1)--(3) from~\S\ref{s:set-uniformity}, where we put $\mathfrak m_Y:=|Y|$:
\begin{enumerate}
\item Assume that $Y=[0,h^\alpha]$ is an interval with $0\leq \alpha\leq 1$. Then the assumption~\eqref{e:NUP-XY-1} is not satisfied. This makes sense in light of the brick counterexample~\eqref{e:brick-ex-1}.
\item Assume that $Y=[0,1]\cap (h^\alpha\mathbb Z+[0,h])$, where $0\leq\alpha\leq 1$. If $\alpha\in \{0\}\sqcup [\tfrac12,1]$, then the assumption~\eqref{e:NUP-XY-1} is not satisfied. If $\alpha\in (0,\tfrac12)$, then
the assumptions~\eqref{e:NUP-XY-1}--\eqref{e:NUP-XY-3} are satisfied as long as
$\zeta\leq \min(\alpha,\tfrac12-\alpha)$ and $\epsilon\leq\tfrac{1}{33}\min(\alpha,\beta)$.
\item Assume that $Y\subset [0,1]$ is a $\delta$-dimensional Ahlfors--David regular set
on scales~$h$ to~1, where $0<\delta<1$. Then 
the assumptions~\eqref{e:NUP-XY-1}--\eqref{e:NUP-XY-3} hold as long as $\zeta\leq \tfrac12\delta(1-\delta)$
and $\epsilon\leq\tfrac{1}{33}\delta\beta$.

In the case when $X,Y\subset[0,1]$ are both $\delta$-dimensional Ahlfors--David regular sets on scales~$h$ to~1, the condition~\eqref{e:NUP-XY-vol} holds when $\delta\leq \tfrac12+\tfrac12\epsilon$ and the bound~\eqref{e:NUP} is better than the volume bound~\eqref{e:uncertainty-volume-2}
when $\delta> \tfrac12-\epsilon$, so Theorem~\ref{t:NUP} is only interesting when $\delta\approx\tfrac12$.
In this case a bound of the type~\eqref{e:NUP}
follows from the fractal uncertainty principles of Dyatlov--Zahl~\cite{hgap}, Bourgain--Dyatlov~\cite{fullgap},
and Dyatlov--Jin~\cite{regfup}. All of these have $\epsilon$ which depends on the regularity constant $C_R$,
while the value of~$\epsilon$ in~\eqref{e:NUP} is a universal constant. However, the work of Bourgain--Dyatlov~\cite{hyperfup} shows~\eqref{e:NUP} with a universal $\epsilon$ when $X,Y$ are limit sets of Schottky groups, which is the case in the applications of fractal uncertainty principle to spectral gaps.

Our recent paper~\cite{fup12} shows an FUP with $\epsilon={1\over 64}$ for $\frac12$-dimensional
Ahlfors--David regular sets. We refer to this paper for a more detailed review of previously known
fractal uncertainty principles.
\end{enumerate}

\subsubsection{Discussion of the nonlinear phase assumption}
\label{s:NUP-discussion}

The nonlinear phase condition~\eqref{e:NUP-curved} in Theorem~\ref{t:NUP} is crucial.
To illustrate this, consider the following counterexample inspired by the Poisson summation formula
(the Fourier transform of the delta-measure on a sublattice of~$\mathbb R$ is the delta-measure on the dual sublattice).
Let
\begin{equation}
  \label{e:bad-phase}
\mathcal U_{\mathcal B}:=\mathbb R^2,\quad 
\Theta(x,y)=-xy,\quad
\chi_{\mathcal B}=1\quad\text{on }[0,1]^2.
\end{equation}
Fix $\alpha\in (0,\tfrac 12)$ and define the sets $X,Y$ as $h$-neighborhoods of arithmetic progressions:
\begin{equation}
  \label{e:bad-XY}
X:=[0,1]\cap(2\pi h^{1-\alpha}\mathbb Z+[0,0.1 h]),\quad
Y:=[0,1]\cap(h^{\alpha}\mathbb Z+[0,0.1 h]).
\end{equation}
These sets satisfy the assumptions~\eqref{e:NUP-XY-1}--\eqref{e:NUP-XY-vol};
see Example~(2) in~\S\ref{s:NUP-examples}.
However, they do not satisfy the uncertainty principle~\eqref{e:NUP}. Indeed, for each $x\in X,y\in Y$ we have
$x\in 2\pi h^{1-\alpha}j+[0,0.1h]$, $y\in h^{\alpha}k+[0,0.1h]$ for some $j,k\in\mathbb Z$ and thus
$$
xy\ \in\ 2\pi hjk+[0,0.2h]\ \subset\ 2\pi h(\mathbb Z+[-\tfrac16,\tfrac16]),
$$
so by Lemma~\ref{l:NUP-lower} we get a lower bound contradicting~\eqref{e:NUP}:
$$
\|\indic_X\mathcal B_h\indic_Y\|_{L^2\to L^2}\geq 0.001.
$$
We now discuss the condition~\eqref{e:NUP-curved} more generally.
We start with the situation when it fails everywhere.
 For simplicity, assume that $\mathcal U_{\mathcal B}=(0,1)^2$.
Let $\Theta$ be a function satisfying~\eqref{e:NUP-nondegenerate} but such that
$$
\partial_x\partial_y\log|\partial_x\partial_y\Theta(x,y)|=0\quad\text{for all }(x,y)\in (0,1)^2.
$$
Then $\Theta$ has the following form:
\begin{equation}
  \label{e:NUP-totally-linear}
\Theta(x,y)=-\varphi_1(x)\varphi_2(y)+a_1(x)+a_2(y)
\end{equation}
for some smooth functions $\varphi_1,\varphi_2,a_1,a_2$ on $(0,1)$ such that $\varphi'_1,\varphi'_2\neq 0$.
The estimate~\eqref{e:NUP} with the phase function~$\Theta$ of the form~\eqref{e:NUP-totally-linear} is equivalent to the same estimate
with the phase function~$\Theta(x,y)=-\varphi_1(x)\varphi_2(y)$, as explained at the end of~\S\ref{s:uncertainty-basics}; by a change of variables we see that it is also equivalent to the same estimate
with the phase function $\Theta(x,y)=-xy$ and the sets $\varphi_1(X),\varphi_2(Y)$.
Thus the counterexample~\eqref{e:bad-XY} applies to any phase function of the form~\eqref{e:NUP-totally-linear};
here we use that the assumptions~\eqref{e:NUP-XY-1}--\eqref{e:NUP-XY-vol}
are diffeomorphism invariant up to changing the constant~$C_0$.

We finally provide a geometric interpretation of the condition~\eqref{e:NUP-curved}.
Let $\varkappa_\Theta:\mathcal U_1\to \mathcal U_2$ be the symplectomorphism generated by~$\Theta$ following~\eqref{e:kappa-Theta-NUP}, where $\mathcal U_1,\mathcal U_2\subset T^*\mathbb R$ are open sets.
Define the \emph{vertical/horizontal foliations} $E_V,E_H$ on $\mathcal U_2$
as follows: the leaves of $E_V$ are the lines $\{x=\const\}$ and the leaves of $E_H$
are the images of the lines $\{y=\const\}$ under $\varkappa_\Theta$.

Let $E_V(x,\xi),E_H(x,\xi)$ be the tangent spaces to the leaves of $E_V,E_H$ through $(x,\xi)$.
That is,
$$
E_V(x,\xi)=\mathbb R\partial_\xi,\quad
E_H(x,\partial_x\Theta(x,y))=\mathbb R(\partial_x+\partial_x^2\Theta(x,y)\partial_\xi).
$$
Note that $E_V(x,\xi)$ and $E_H(x,\xi)$ are transverse to each other, that is $T\mathcal U_2=E_V\oplus E_H$.

Using the foliations $E_V,E_H$ and the standard symplectic form $\omega=d\xi\wedge dx$,
define the following Lorentzian metric $g_{\fol}$ whose null directions are given by $E_V,E_H$:
for all $v,w\in T_{(x,\xi)}\mathcal U_2$,
\begin{equation}
  \label{e:g-fol-def}
\begin{gathered}
g_{\fol}(x,\xi)(v,w):=\omega(v_H,w_V)+\omega(w_H,v_V)\quad\text{where}
\\
v=v_V+v_H,\
w=w_V+w_H,\
v_V,w_V\in E_V(x,\xi),\
v_H,w_H\in E_H(x,\xi).
\end{gathered}
\end{equation}
Let us use the coordinates $(x,y)$ on $\mathcal U_2$, where $\xi=\partial_x\Theta(x,y)$. Then
$$
E_V=\mathbb R\partial_y,\quad
E_H=\mathbb R\partial_x,\quad
\omega=\partial_x\partial_y\Theta(x,y)\,dy\wedge dx,
$$
so we have
$$
g_{\fol}=-2\partial_x\partial_y\Theta(x,y)\,dxdy.
$$
The Gauss curvature of this metric is given by
\begin{equation}
  \label{e:g-fol-K}
K_{g_{\fol}}(x,y)={\partial_x\partial_y\log|\partial_x\partial_y\Theta(x,y)|\over \partial_x\partial_y\Theta(x,y)}.
\end{equation}
Indeed, if we denote by $\nabla$ the Levi-Civita connection of $g_{\fol}$, then
by the Koszul formula
$$
\nabla_{\partial_x}\partial_y=\nabla_{\partial_y}\partial_x=0,\quad
\nabla_{\partial_x}\partial_x=\partial_x\log|\partial_x\partial_y\Theta|\, \partial_x,\quad
\nabla_{\partial_y}\partial_y=\partial_y\log|\partial_x\partial_y\Theta|\, \partial_y
$$
which gives the Riemann curvature tensor
$$
R(\partial_x,\partial_y)\partial_y=\partial_x\partial_y\log|\partial_x\partial_y\Theta|\,\partial_y
$$
and implies~\eqref{e:g-fol-K}.

We see now that the condition~\eqref{e:NUP-curved} is equivalent to the metric $g_{\fol}$ having nonvanishing curvature. Note that the metric $g_{\fol}$ was defined using only the foliations $E_V,E_H$ and the symplectic form~$\omega$.
This definition in fact appears in the theory of \emph{bi-Lagrangian}, or \emph{para-K\"ahler}, structures,
see for example~\cite{Etayo-bilagrangian}.

\subsection{Reducing to an additive energy bound}
\label{s:NUP-reductions}

In this section we reduce Theorem~\ref{t:NUP} to an additive energy estimate,
Proposition~\ref{l:NUP-additive} below.

\subsubsection{Basic reductions}
\label{s:NUP-basic-reductions}

To simplify the notation below, we will show Theorem~\ref{t:NUP} in the following special case:
\begin{gather}
  \label{e:NUPred}
\Theta\in C^\infty([0,1]^2;\mathbb R),\quad
\chi_{\mathcal B}\in \CIc((0.1,0.9)^2);
\\
\label{e:Theta-mix-12}
1\leq \partial_x\partial_y\Theta(x,y)\leq 2;
\\
  \label{e:Theta-c-nlin}
|\partial_x\partial_y\log\partial_x\partial_y\Theta(x,y)|\geq c_{\nlin}>0.
\end{gather}
The constants $C$ below will depend on some $(\zeta,\epsilon)$-dependent $C^\infty$-seminorms
of $\Theta,\chi_{\mathcal B}$ and on the constants $C_0,c_{\nlin}$.
We will also assume that $h$ is small enough depending on the same data. Below $\Theta$
will always refer to the same phase function, but $X,Y$ can refer to various sets
(not necessarily the ones to which Theorem~\ref{t:NUP} applies).

The general case reduces to this special case as follows. For each $(x,y)\in \mathcal K_{\mathcal B}$,
there exists a cube $\mathcal U_{(x,y)}\subset\mathcal U_{\mathcal B}$ centered at $(x,y)$ and such that
$c\leq |\partial_x\partial_y\Theta|\leq 2c$ on $\mathcal U_{(x,y)}$ for some $c>0$.
Let $0.8\,\mathcal U_{(x,y)}$ be the cube centered at $(x,y)$
whose side length is $0.8$ times that of $\mathcal U_{(x,y)}$.
Taking a partition of unity, we reduce to the case when $\supp\chi_{\mathcal B}$ lies
in the interior of $0.8\,\mathcal U_{(x,y)}$ for some $(x,y)$.
Making an affine change of variables in $x,y$, we may assume that
$\mathcal U_{(x,y)}=[0,1]^2$ and $\supp\chi_{\mathcal B}\subset (0.1,0.9)^2$. Making the additional change of variables $x\mapsto 1-x$
if necesssary, we may also assume that $\partial_x\partial_y\Theta>0$ on $[0,1]^2$,
and then $c\leq \partial_x\partial_y\Theta\leq 2c$ for some $c>0$. We can
finally make the change $\Theta\mapsto c^{-1}\Theta$, $h\mapsto c^{-1}h$
to ensure that $1\leq\partial_x\partial_y\Theta\leq 2$.
The resulting changes in $X,Y,h$ preserve the assumptions~\eqref{e:NUP-XY-1}--\eqref{e:NUP-XY-vol}
up to a change in the constant $C_0$, as follows from~\eqref{e:r-cover-monotone}, \eqref{e:M-r-monotone},
and~\eqref{e:volume-nbhd}.

We may also replace the sets $X,Y$ in Theorem~\ref{t:NUP} by their intersections with the interval $[0.1,0.9]$, since the statement~\eqref{e:NUP}
does not change because of the support condition~\eqref{e:NUPred}. So we may assume that
\begin{equation}
  \label{e:XY-contained}
X,Y\ \subset\ [0.1,0.9].
\end{equation}

\subsubsection{The additive energy bound}

To state the additive energy bound to which we will reduce Theorem~\ref{t:NUP},
we need to introduce some notation.
Denote
\begin{equation}
  \label{e:vec-y}
\vec y=(y_1,y_2,y'_1,y'_2)\in \mathbb R^4,\quad
d\vec y:=dy_1dy_2dy'_1dy'_2.
\end{equation}
Define the following function (where $x\in [0,1]$, $\vec y\in [0,1]^4$)
\begin{equation}
  \label{e:Theta-vec-def}
\vec\Theta(x,\vec y):=\Theta(x,y_1)+\Theta(x,y_2)
-\Theta(x,y'_1)-\Theta(x,y'_2).
\end{equation}
Next, for $r>0$ define the set
\begin{equation}
  \label{e:E-r-def}
\mathcal E_r:=\{(x,\vec y)\in [0,1]^5\colon
|\partial_x\vec\Theta(x,\vec y)|\leq r\}.
\end{equation}
For sets $X,Y\subset\mathbb R$, define the additive energy quantity
\begin{equation}
  \label{e:E-r-2-def}
\mathbf E_r(X\xrightarrow\Theta Y)
:=
{\big|(X\times Y^4)\cap \mathcal E_r\big|\over |X|\cdot|Y|^3 \mathbf M_r(Y)}.
\end{equation}
By~\eqref{e:Theta-mix-12}, for each $(x,y_1,y_2,y'_1)\in [0,1]^4$, the set
of $y'_2\in [0,1]$ such that $(x,\vec y)\in\mathcal E_r$ is contained in a $2r$-interval.
It follows from~\eqref{e:M-r-monotone} that
\begin{equation}
  \label{e:E-r-basic}
\mathbf E_r(X\xrightarrow\Theta Y)\leq 2.
\end{equation}
We can now state our additive energy bound. It is proved in~\S\ref{s:NUP-additive-proof} below.
\begin{proposition}[Additive energy bound]
\label{l:NUP-additive}
For each $\zeta\in (0,0.1]$ there exists $\beta\in (0, \zeta/2]$ such that for all $\epsilon\in (0,0.1\beta]$,
any phase function $\Theta$ satisfying~\eqref{e:NUPred}--\eqref{e:Theta-c-nlin},
and any~$h>0$ small enough depending only on $\zeta,\epsilon,c_{\nlin}$, and the $C^3$ norm
of $\Theta$, the following holds.
Assume that $X,Y\subset [0.1,0.9]$ are two sets of positive measure satisfying the following
conditions:
\begin{align}
  \label{e:NUPa-Y-uni}
|Y|_{h^{\frac12}}\mathbf M_{h^{\frac12}}(Y) &\leq h^{-\epsilon/2}|Y|,
\\
  \label{e:NUPa-Y-noncon1}
\mathbf M_{|Y(h^{\frac12})|}(Y(h^{\frac12})) &\leq
h^{\zeta/2}|Y(h^{\frac12})|,
\\
  \label{e:NUPa-Y-noncon2}
\mathbf M_{h^{\beta}}(Y) &\leq h^{27\epsilon}|Y|,
\\
  \label{e:NUPa-X-noncon}
\mathbf M_r(X(h^{\frac12})) &\leq h^{-2\beta}r^\zeta |X(h^{\frac12})|\quad\text{for all }r\in [h^{\frac12},1].
\end{align}
Then we have the following estimate:
\begin{equation}
  \label{e:NUP-additive-conc}
\mathbf E_{h^{\frac12-\epsilon}}(X(h^{\frac12})\xrightarrow\Theta Y)\leq h^{26\epsilon}.
\end{equation}
\end{proposition}

We briefly discuss how the conditions of Theorem~\ref{t:NUP} are related to those of Proposition~\ref{l:NUP-additive}.
In~\S\ref{s:NUP-reduction-proof} below, we reduce to the case when the sets $X,Y$
are uniform (see~\S\ref{s:set-uniformity}); condition~\eqref{e:NUPa-Y-uni}
follows from uniformity. Conditions~\eqref{e:NUPa-Y-noncon1}, \eqref{e:NUPa-Y-noncon2}
are variants of conditions~\eqref{e:NUP-XY-1}, \eqref{e:NUP-XY-2}.
Condition~\eqref{e:NUPa-X-noncon} on $X$ follows from the condition~\eqref{e:NUP-XY-3}
on~$Y$ (with $r$ replaced by~$h/r$) and uniformity: if \eqref{e:NUP-XY-3} holds but~\eqref{e:NUPa-X-noncon} fails,
then the refined volume bound of Lemma~\ref{l:L2-intermediate} gives the $L^2\to L^2$ estimate~\eqref{e:NUP}.
Finally, condition~\eqref{e:NUP-XY-vol} does not appear explicitly in Proposition~\ref{l:NUP-additive}
but it is needed to show that the additive energy bound~\eqref{e:NUP-additive-conc} implies
the $L^2\to L^2$ estimate~\eqref{e:NUP}.
The nonlinear phase condition~\eqref{e:Theta-c-nlin}
is not needed in the reduction of Theorem~\ref{t:NUP} to Proposition~\ref{l:NUP-additive};
it is only used in the proof of Proposition~\ref{l:NUP-additive}.

Let us also explain how Proposition~\ref{l:NUP-additive} fails for the example discussed at the beginning of~\S\ref{s:NUP-discussion}.
If $\Theta(x,y)=-xy$ then $\vec\Theta(x,\vec y)=-x(y_1+y_2-y'_1-y'_2)$ and thus
$$
\mathcal E_r=[0,1]\times \{\vec y\in [0,1]^4\colon |y_1+y_2-y'_1-y'_2|\leq r\}. 
$$
Fix $\alpha\in (0,\tfrac12)$ and put $X:=[0,1]\cap (2\pi h^{1-\alpha}\mathbb Z+[0,0.1h])$,
$Y:=[0,1]\cap (h^\alpha\mathbb Z+[0,0.1h])$. These sets satisfy the assumptions~\eqref{e:NUPa-Y-uni}--\eqref{e:NUPa-X-noncon} for $h$ small enough as long as~$\zeta< 2\min(\alpha,\tfrac12-\alpha)$ and 
$\epsilon<\tfrac{1}{27}\min(\alpha,\beta)$; see~\S\ref{s:set-uniformity}. 

To estimate the additive energy $\mathbf E_r(X(h^{\frac12})\xrightarrow\Theta Y)$ from below, define the sets
$$
\begin{aligned}
Y_1\, & := [0,\tfrac13]\cap (h^\alpha\mathbb Z+0.1h[0,\tfrac13]),
\\
Y_2\, & := [\tfrac23,1]\cap (h^\alpha\mathbb Z+0.1h[\tfrac23,1]),
\\
Y_3\, & := [\tfrac13,\tfrac23] \cap (h^\alpha \mathbb Z+0.1h [\tfrac13,\tfrac23]).
\end{aligned}
$$
Then $Y_1,Y_2,Y_3,Y_1+Y_2-Y_3\subset Y$, so for each $(y_1,y_2,y'_1)\in Y_1\times Y_2\times Y_3$,
the set of $y'_2$ such that $|y_1+y_2-y'_1-y'_2|\leq r$ is
a $2r$-interval centered at a point in~$Y$. 
By the uniformity of~$Y$ discussed in~\S\ref{s:set-uniformity}, 
when $r\in [h,1]$ the measure of the intersection of this interval with $Y$ is bounded below by
$0.01\mathbf M_r(Y)$ and thus
$$
\big|(X(h^{\frac12})\times Y^4)\cap \mathcal E_r\big|\geq 0.01|X(h^{\frac12})|\cdot |Y_1|\cdot |Y_2|\cdot |Y_3|\mathbf M_r(Y).
$$
Since $|Y_1|,|Y_2|,|Y_3|\geq 0.1|Y|$, it follows that
$$
\mathbf E_r(X(h^{\frac12})\xrightarrow\Theta Y)\geq 10^{-5}\quad\text{for all }r\in[h,1]
$$
which contradicts~\eqref{e:NUP-additive-conc} for $h$ small enough.

\subsubsection{Nonstationary phase}

We will use the following nonstationary phase lemma, similar to~\cite[Lemma~5.1]{hgap}.
\begin{lemma}
  \label{l:ibp}
Assume that $\mathcal U\subset \mathbb R$ is an open set,
$\Phi\in C^\infty(\mathcal U;\mathbb R)$, $a\in \CIc(\mathcal U)$, and
the following inequalities hold for some positive numbers $R_1, R_2,(C_j)_{j\geq 0}$:
\begin{align}
  \label{e:ibpa-1}
|\partial_x\Phi(x)|&\geq C_0^{-1}R_1\quad\text{for all }x\in\supp a;
\\
  \label{e:ibpa-2}
|\partial_x^{j+1}\Phi(x)|&\leq C_j R_1R_2^j\quad\text{for all }x\in\supp a,\
j\geq 1;
\\
  \label{e:ibpa-3}
|\partial_x^j a(x)|&\leq C_j R_2^j\quad\text{for all }x\in \mathcal U,\
j\geq 0.
\end{align}
Then we have for each $N\geq 0$
\begin{equation}
  \label{e:ibp}
\bigg|\int_{\mathbb R} e^{i\Phi(x)}a(x)\,dx\bigg|\leq C'_N\Big({R_2\over R_1}\Big)^N|\supp a|
\end{equation}
where the constant $C'_N$ depends only on $N,C_0,C_1,\dots,C_N$.
\end{lemma}
\begin{proof}
Consider the following first order differential operator on $\mathcal U$:
$$
\mathcal L:=-i(\partial_x\Phi(x))^{-1}\partial_x.
$$
Then $e^{i\Phi}=\mathcal Le^{i\Phi}$, so we can integrate by parts $N$ times to get
$$
\int_{\mathbb R} e^{i\Phi(x)}a(x)\,dx
=\int_{\mathbb R} \big(\mathcal L^N e^{i\Phi(x)}\big)a(x)\,dx
=\int_{\mathbb R} e^{i\Phi(x)}
(\mathcal L^t)^Na(x)\,dx
$$
where $\mathcal L^t$ is the transpose operator:
$$
\mathcal L^tf(x)=i\partial_x\bigg({f(x)\over\partial_x\Phi(x)}\bigg).
$$
It follows that
\begin{equation}
  \label{e:ibp-1}
\bigg|\int_{\mathbb R} e^{i\Phi(x)}a(x)\,dx\bigg|
\leq |\supp a|\cdot \sup_{x\in\supp a}|(\mathcal L^t)^Na(x)|.
\end{equation}
The function $(\mathcal L^t)^ Na(x)$ is a linear combination with constant
coefficients of terms of the form
$$
(\partial_x\Phi(x))^{-N-K}\Big(\prod_{k=1}^K {\partial_x^{j_k+1}\Phi(x)}\Big)\partial_x^{j_0}a(x)
$$
where $j_0\geq 0$, $j_1,\dots,j_K\geq 1$, and $j_0+j_1+\dots+j_K= N$.
By~\eqref{e:ibpa-1}--\eqref{e:ibpa-3}, the sup-norm of this expression
over $x\in\supp a$ is bounded above by
$$
C'_N R_1^{-N-K}\Big(\prod_{k=1}^K R_1R_2^{j_k}\Big)R_2^{j_0}=C'_N\Big({R_2\over R_1}\Big)^N.
$$
Together with~\eqref{e:ibp-1} this gives~\eqref{e:ibp}.
\end{proof}

\subsubsection{Volume bounds and orthogonality}

We show here a refined version of~\eqref{e:B-NUP-L2} and~\eqref{e:uncertainty-volume-2}
which argues at some intermediate scale. As before, we assume that
the phase function $\Theta$ and the amplitude $\chi_{\mathcal B}$ in the definition~\eqref{e:B-NUP-def}
of the operator~$\mathcal B$ satisfy the conditions~\eqref{e:NUPred} and~\eqref{e:Theta-mix-12}.
\begin{lemma}
  \label{l:L2-intermediate}
We have for all $h\in (0,1]$, $r\in [h,1]$, and $X,Y\subset \mathbb R$
\begin{equation}
  \label{e:L2-intermediate}
\|\indic_X\mathcal B\indic_Y\|_{L^2\to L^2}\leq Ch^{-\frac12}\sqrt{\mathbf M_r(X)\,|Y(h/r)|}.
\end{equation}
Here the constant $C$ depends only on the $C^4$ norm of $\Theta$ and the $C^2$ norm of $\chi_{\mathcal B}$.
\end{lemma}
\begin{remark}
\label{l:L2-intermediate-remark}
If $r=1$ and we assume that $Y$ is a union of~$h$-intervals, then~\eqref{e:L2-intermediate}
gives the volume bound~\eqref{e:uncertainty-volume-2}.
If $r=h$, then~\eqref{e:L2-intermediate} with $X=Y=[0,1]$ implies the bound~\eqref{e:B-NUP-L2}.
\end{remark}
\begin{proof}
Because of the support condition~\eqref{e:NUPred}, without loss of generality we may assume that $X,Y\subset[0,1]$.

\noindent 1. Fix a cutoff function
$$
\chi\in \CIc((-1,1);\mathbb R),\quad
\int_{\mathbb R}\chi(y)\,dy=1.
$$
Define the cutoff adapted to the set $Y$
$$
\chi_Y\in\CIc(\mathbb R),\quad
\chi_Y(y)= 2h^{-1}r\int_{Y(h/(2r))}\chi(2h^{-1}r(y-y'))\,dy'.
$$
Then
$$
0\leq\chi_Y\leq 1,\quad
\supp\chi_Y\subset Y(h/r),\quad
\chi_Y=1\quad\text{on }Y,
$$
and we have the derivative bounds with some universal constants $C_j$
\begin{equation}
  \label{e:L2i-1}
\sup|\partial^j_y\chi_Y|\leq C_j(r/h)^j\quad\text{for all }j\geq 0.
\end{equation}
Now, we have
$$
\|\indic_X\mathcal B\indic_Y\|_{L^2\to L^2}^2
\leq \|\indic_X\mathcal B\chi_Y\|_{L^2\to L^2}^2
=\|\indic_X\mathcal B\chi_Y^2\mathcal B^*\indic_X\|_{L^2\to L^2}.
$$
We compute
$$
\begin{aligned}
\mathcal B\chi_Y^2\mathcal B^*f(x)&=\int_{\mathbb R}K(x,x')f(x')\,dx'\quad\text{where}
\\
K(x,x')&=(2\pi h)^{-1}\int_{\mathbb R}e^{{i\over h}(\Theta(x,y)-\Theta(x',y))}
\chi_Y(y)^2\chi_{\mathcal B}(x,y)\overline{\chi_{\mathcal B}(x',y)}\,dy.
\end{aligned}
$$
By Schur's bound (see e.g.~\cite[Lemma~18.1.12]{Hormander3}), we have
\begin{equation}
  \label{e:L2i-2}
\|\indic_X\mathcal B\indic_Y\|_{L^2\to L^2}^2\leq \sup_{x\in X}\int_{X}
|K(x,x')|\,dx'.
\end{equation}

\noindent 2. We claim that for all $x,x'\in [0,1]$
\begin{equation}
  \label{e:L2i-3}
|K(x,x')|\leq Ch^{-1}\bigg(1+{|x-x'|\over r}\bigg)^{-2}|Y(h/r)|.
\end{equation}
If $|x-x'|\leq r$, this is immediate since $\supp\chi_Y\subset Y(h/r)$. Assume now that $|x-x'|\geq r$.
We use Lemma~\ref{l:ibp} with
$$
\mathcal U:=(0,1),\quad
\Phi(y):={\Theta(x,y)-\Theta(x',y)\over h},\quad
a(y):=\chi_Y(y)^2\chi_{\mathcal B}(x,y)\overline{\chi_{\mathcal B}(x',y)}.
$$
We have for all $j\geq 0$
$$
\partial_y^{j+1}\Phi(y)=h^{-1}\int_{x'}^x \partial_x\partial_y^{j+1}\Theta(x'',y)\,dx'',
$$
so by~\eqref{e:Theta-mix-12} and~\eqref{e:L2i-1} we have for all $y\in [0,1]$ and $j\geq 0$
$$
|\partial_y\Phi(y)|\geq {|x-x'|\over h},\quad
|\partial_y^{j+1}\Phi(y)|\leq C_j{|x-x'|\over h},\quad
|\partial_y^ja(y)|\leq C_j \Big({r\over h}\Big)^j.
$$
The bounds~\eqref{e:ibpa-1}--\eqref{e:ibpa-3} hold with
$$
R_1:={|x-x'|\over h},\quad
R_2:={r\over h}.
$$
Since $\supp a\subset\supp \chi_Y\subset Y(h/r)$, we see that~\eqref{e:ibp} with $N:=2$ implies~\eqref{e:L2i-3} as needed. An examination of the above argument shows that the constant $C$ in~\eqref{e:L2i-3} only depends on the $C^4$ norm of~$\Theta$ and the $C^2$ norm of~$\chi_{\mathcal B}$.

\noindent 3. Fix any $x\in [0,1]$ and consider the dyadic intervals
$$
J_\ell:=[x-2^\ell r,x+2^\ell r],\quad
\ell\geq 0.
$$
Then, denoting $J_{-1}:=\emptyset$,
$$
\begin{aligned}
\int_X |K(x,x')|\,dx'
&=\sum_{\ell\geq 0}\int_{X\cap (J_\ell\setminus J_{\ell-1})}|K(x,x')|\,dx'
\\
&\leq Ch^{-1}\sum_{\ell\geq 0} 2^{-2\ell}|X\cap J_\ell|\cdot |Y(h/r)|
\\
&\leq Ch^{-1}\mathbf M_r(X)\,|Y(h/r)|
\end{aligned}
$$
where the first inequality follows from~\eqref{e:L2i-3} and
in the second inequality we use that $|X\cap J_\ell|\leq \mathbf M_{2^{\ell+1}r}(X)\leq 2^{\ell+1}\mathbf M_r(X)$ by~\eqref{e:M-r-monotone}. Together with~\eqref{e:L2i-2} this gives~\eqref{e:L2-intermediate}.
\end{proof}

\subsubsection{An $L^4$ estimate}

We now bound the norm of $\indic_X\mathcal B\indic_Y$ in terms of the energy
quantity $\mathbf E_r$ introduced in~\eqref{e:E-r-2-def}. For this it is convenient
to consider the $L^\infty\to L^4$ norm; later in~\eqref{e:L2-L2} we interpolate between
this and the $L^1\to L^\infty$ norm to bound the $L^2\to L^2$ norm.

We start with the case when $X$ is replaced by a $2h^{\frac12}$-interval:
\begin{lemma}
  \label{l:Linfty-L4-interval}
Let $x_0\in [0,1]$, $Y\subset [0,1]$, and put $J:=[x_0-h^{\frac12},x_0+h^{\frac12}]$. Then
we have for all $\epsilon>0$ and $N$
\begin{equation}
  \label{e:Linfty-L4-interval}
\|\indic_J \mathcal B\indic_Y\|_{L^\infty\to L^4}^4\leq
Ch^{-\frac32}\big|\{\vec y\in Y^4\colon (x_0,\vec y)\in \mathcal E_{h^{\frac12-\epsilon}}\}\big|+Ch^N.
\end{equation}
Here the constant $C$ depends only on some $(N,\epsilon)$-dependent $C^\infty$ seminorms
of $\Theta,\chi_{\mathcal B}$.
\end{lemma}
\begin{proof}
\noindent 1. Fix a cutoff function
$$
\chi\in \CIc((-2,2);[0,1]),\quad
\chi=1\quad\text{on }[-1,1]
$$
and define the function
$$
\chi_J\in\CIc(\mathbb R),\quad
\chi_J(x)=\chi\Big({x-x_0\over h^{\frac12}}\Big).
$$
Note that $\chi_J=1$ on $J$.

Take $f\in L^\infty(\mathbb R)$. Then, using the notation~\eqref{e:vec-y} and~\eqref{e:Theta-vec-def},
\begin{equation}
  \label{e:Li4i-0}
\begin{aligned}
\|\indic_J\mathcal B\indic_Y f\|_{L^4}^4
&\leq
\int_{\mathbb R}\chi_J(x)|\mathcal B\indic_Y f(x)|^4\,dx
\\
&=\int_{\mathbb R}\chi_J(x)\mathcal B\indic_Y f(x)\mathcal B\indic_Y f(x)
\overline{\mathcal B\indic_Y f(x)\mathcal B\indic_Y f(x)}\,dx
\\
&=\int_{Y^4}F_J(\vec y)f(y_1)f(y_2)\overline{f(y'_1)f(y'_2)}\,d\vec y
\end{aligned}
\end{equation}
where
$$
\begin{aligned}
F_J(\vec y)\,&:=
(2\pi h)^{-2}\int_{\mathbb R}e^{{i\over h}\vec\Theta(x,\vec y)}\vec\chi_J(x,\vec y)\,dx,
\\
\vec\chi_J(x,\vec y)\,&:=
\chi_J(x)\chi_{\mathcal B}(x,y_1)\chi_{\mathcal B}(x,y_2)
\overline{\chi_{\mathcal B}(x,y'_1)\chi_{\mathcal B}(x,y'_2)}.
\end{aligned}
$$

\noindent 2. We claim that for all $\vec y\in [0,1]^4$ and $\widetilde N$
\begin{equation}
  \label{e:Li4i-1}
|F_J(\vec y)|\leq C_{\widetilde N} h^{-\frac32}\bigg(1+{|\partial_x\vec\Theta(x_0,\vec y)|\over h^{\frac12}}\bigg)^{-\widetilde N}.
\end{equation}
For that we use Lemma~\ref{l:ibp} with
$$
\mathcal U:=(0,1),\quad
\Phi(x):={\vec\Theta(x,\vec y)\over h},\quad
a(x):=\vec\chi_J(x,\vec y).
$$
Note that $\supp a\subset \supp\chi_J\subset [x_0-2h^{\frac12},x_0+2h^{\frac12}]$,
so $|\supp a|\leq 4h^{\frac12}$ and for all $x\in\supp a$ we have
$$
|\partial_x\vec\Theta(x,\vec y)-\partial_x\vec\Theta(x_0,\vec y)|\leq \widetilde Ch^{\frac12},\quad
\widetilde C:=\max(8\sup|\partial_x^2\Theta|,1).
$$
If $|\partial_x\vec\Theta(x_0,\vec y)|\leq 2\widetilde Ch^{\frac12}$, then~\eqref{e:Li4i-1} is immediate.
Assume now that $|\partial_x\vec\Theta(x_0,\vec y)|\geq 2\widetilde Ch^{\frac12}$. We have for all $x\in \supp a$ and $j\geq 0$
$$
|\partial_x\Phi(x)|\geq {|\partial_x\vec\Theta(x_0,\vec y)|\over 2h},\quad
|\partial_x^{j+1}\Phi(x)|\leq {C_j\over h},\quad
|\partial_x^j a(x)|\leq C_j h^{-\frac12 j}.
$$
The bounds~\eqref{e:ibpa-1}--\eqref{e:ibpa-3} hold with
$$
R_1:={|\partial_x\vec\Theta(x_0,\vec y)|\over h},\quad
R_2:=h^{-\frac12}.
$$
Then~\eqref{e:ibp} gives~\eqref{e:Li4i-1} as needed.

\noindent 3. By~\eqref{e:Li4i-0}, we have

$$
\begin{aligned}
{\|\indic_J\mathcal B\indic_Y f\|_{L^4}^4\over \|f\|_{L^\infty}^4}
\leq 
\int_{Y^4}|F_J(\vec y)|\,d\vec y
\leq &\,
\big|\{\vec y\in Y^4\colon |\partial_x\vec\Theta(x_0,\vec y)|\leq h^{\frac12-\epsilon}\}\big|\sup_{\vec y\in [0,1]^4}|F_J(\vec y)|
\\
&+ \sup_{\vec y\in [0,1]^4\colon |\partial_x\vec\Theta(x_0,\vec y)|\geq h^{\frac12-\epsilon}}
|F_J(\vec y)|.
\end{aligned}
$$
By~\eqref{e:Li4i-1}, we see that the first supremum on the right-hand side
is bounded above by $Ch^{-\frac32}$ and the second one is bounded by
$C_{\widetilde N}h^{\epsilon\widetilde N-\frac32}\leq Ch^N$,
where we take $\widetilde N\geq\epsilon^{-1}(N+\frac32)$. This finishes the proof of~\eqref{e:Linfty-L4-interval}.
\end{proof}
\begin{remark}
If we instead choose $J$ to have length $2h^{\alpha}$ for some $\alpha\in [0,1]$, then
the bounds~\eqref{e:ibpa-1}--\eqref{e:ibpa-3} hold as long as $|\partial_x\vec\Theta(x_0,\vec y)|\geq 2\widetilde C h^\alpha$, with $R_1:=h^{-1}|\partial_x\Theta(x_0,\vec y)|$ and $R_2:=h^{-\alpha}$.
The estimate~\eqref{e:ibp} only gives an improvement over the $N=0$ case
when $R_2\ll R_1$, that is $|\partial_x\vec\Theta(x_0,\vec y)|\gg h^{1-\alpha}$.
The choice of $\alpha:=\tfrac12$ in Lemma~\ref{l:Linfty-L4-interval} is helpful because the two lower bounds
on $|\partial_x\vec\Theta(x_0,\vec y)|$ both feature the same power of~$h$, namely $h^{\frac12}$.
\end{remark}
We now estimate the norm of $\indic_X\mathcal B\indic_Y$ by summing Lemma~\ref{l:Linfty-L4-interval}
over a collection of $h^{\frac12}$-intervals covering~$X$:
\begin{lemma}
  \label{l:Linfty-L4}
Let $X,Y\subset [0,1]$ have positive measure. Then
for all $\epsilon>0$ and~$N$
\begin{gather}
  \label{e:Linfty-L4}
\begin{aligned}
\|\indic_X\mathcal B\indic_Y\|_{L^\infty\to L^4}
\leq&\,
Ch^{-\frac12}|X(h^{\frac12})|^{\frac14}\,|Y|^{\frac34}\mathbf M_{h^{\frac12-\epsilon}}(Y)^{\frac14}
\mathbf E_{h^{\frac12-\epsilon}}(X(h^{\frac12})\xrightarrow\Theta Y)^{\frac14}\\
&+Ch^N,
\end{aligned}
\\
  \label{e:L2-L2}
\begin{aligned}
\|\indic_X\mathcal B\indic_Y\|_{L^2\to L^2}
\leq&\,
Ch^{-\frac12}|X|^{\frac38}\,
|Y|^{\frac38}|X(h^{\frac12})|^{\frac 18}\,\mathbf M_{h^{\frac12-\epsilon}}(Y)^{\frac18}
\mathbf E_{h^{\frac12-\epsilon}}(X(h^{\frac12})\xrightarrow\Theta Y)^{\frac18}\\
&+Ch^N.
\end{aligned}
\end{gather}
Here the constant $C$ depends only on some $(N,\epsilon)$-dependent $C^\infty$ seminorms of $\Theta,\chi_{\mathcal B}$.
\end{lemma}
\begin{remark}
Let us come back to the example discussed at the beginning of~\S\ref{s:NUP-discussion}.
Fix $\alpha\in (0,\tfrac12-\epsilon]$ and put $X:=[0,1]\cap (2\pi h^{1-\alpha}\mathbb Z+[0,0.1h])$,
$Y:=[0,1]\cap (h^\alpha\mathbb Z+[0,0.1h])$. Then
$|X|\sim h^\alpha$, $|Y|\sim h^{1-\alpha}$, $|X(h^{\frac12})|\sim 1$,
 and $\mathbf M_{h^{\frac12-\epsilon}}(Y)\sim h$.
Thus the first term on the right-hand side of~\eqref{e:L2-L2} is
$\sim \mathbf E_{h^{\frac12-\epsilon}}(X(h^{\frac12})\xrightarrow\Theta Y)^{\frac18}$.
That is, any improvement over the basic additive energy bound~\eqref{e:E-r-basic}
gives decay of the norm $\|\indic_X\mathcal B\indic_Y\|_{L^2\to L^2}$.
\end{remark}
\begin{proof}
1. Let $x_1,\dots,x_L\in X$ be a maximal $\frac12h^{\frac12}$-separated set
and define
$$
J_\ell:=[x_\ell-\tfrac12h^{\frac12},x_\ell+\tfrac12h^{\frac12}].
$$
Then $X\subset\bigcup_{\ell=1}^L J_\ell$ and thus
$$
\|\indic_X\mathcal B\indic_Y\|_{L^\infty\to L^4}^4
\leq \sum_{\ell=1}^L\|\indic_{J_\ell}\mathcal B\indic_Y\|_{L^\infty\to L^4}^4.
$$
For each $x\in J_\ell\cap [0,1]$, we have $J_\ell\subset [x-h^{\frac12},x+h^{\frac12}]$ and thus
by Lemma~\ref{l:Linfty-L4-interval}
$$
\|\indic_{J_\ell} \mathcal B\indic_Y\|_{L^\infty\to L^4}^4\leq
Ch^{-\frac32}\big|\{\vec y\in Y^4\colon (x,\vec y)\in \mathcal E_{h^{\frac12-\epsilon}}\}\big|
+Ch^{4N+1}.
$$
Integrating over $x\in J_\ell\cap [0,1]$, we get
$$
\|\indic_{J_\ell} \mathcal B\indic_Y\|_{L^\infty\to L^4}^4\leq
Ch^{-2}\big|
(J_\ell\times Y^4)\cap \mathcal E_{h^{\frac12-\epsilon}}
\big|+Ch^{4N+1}.
$$
Summing over $\ell$, and using that each point lies in at most 3 intervals $J_\ell$ and
$\bigcup_{\ell=1}^L J_\ell\subset X(h^{\frac12})$,
we then have
$$
\|\indic_X\mathcal B\indic_Y\|_{L^\infty\to L^4}^4
\leq
Ch^{-2}\big|
(X(h^{\frac12})\times Y^4)\cap \mathcal E_{h^{\frac12-\epsilon}}
\big|
+Ch^{4N}.
$$
Recalling the definition~\eqref{e:E-r-2-def}, we see that this implies~\eqref{e:Linfty-L4}.

\noindent 2. By H\"older's inequality, we have the bounds
\begin{align}
  \label{e:Linfty-L2}
\|\indic_X\mathcal B\indic_Y\|_{L^\infty\to L^2}
&\leq
|X|^{\frac14}\,\|\indic_X\mathcal B\indic_Y\|_{L^\infty\to L^4},
\\
  \label{e:L1-L2}
\|\indic_X\mathcal B\indic_Y\|_{L^1\to L^2}
&\leq
|X|^{\frac12}\,\|\indic_X\mathcal B\indic_Y\|_{L^1\to L^\infty}.
\end{align}
Interpolating between these gives
$$
\|\indic_X\mathcal B\indic_Y\|_{L^2\to L^2}
\leq
|X|^{\frac38}\,
\|\indic_X\mathcal B\indic_Y\|_{L^\infty\to L^4}^{\frac12}\,
\|\indic_X\mathcal B\indic_Y\|_{L^1\to L^\infty}^{\frac12}.
$$
Together with~\eqref{e:Linfty-L4} and~\eqref{e:L1-Linfty} this gives~\eqref{e:L2-L2}.
\end{proof}

\subsubsection{Proof of the reduction}
\label{s:NUP-reduction-proof}

We are now ready to show that Proposition~\ref{l:NUP-additive} implies Theorem~\ref{t:NUP}.
We first give a proof of a weaker version of this theorem,
which assumes uniformity of $X,Y$:
\begin{proposition}[NUP, uniform case]
  \label{l:NUP-reduced-2}
For each $\zeta\in(0,0.1]$ there exists $\beta\in(0,\zeta/2]$ such that for all $\epsilon\in (0,0.1\beta]$
and $h\in (0,h_0]$ the following holds.
Assume that:
\begin{itemize}
\item
$\mathcal B$ is an operator of the form~\eqref{e:B-NUP-def},
satisfying the assumptions~\eqref{e:NUPred}--\eqref{e:Theta-c-nlin};
\item
$C_0>0$, $\mathfrak m_Y>0$, and $X,Y\subset [0.1,0.9]$ are two sets satisfying the conditions~\eqref{e:NUP-XY-1}--\eqref{e:NUP-XY-vol};
\item
$X,Y$ are $r_1(h)$-uniform on scales $h$ to~1 in the sense
of Definition~\ref{d:set-uniformity}, where $r_1(h)$ is given in Lemma~\ref{l:uniform-set-decomposition}.
\end{itemize}
Then we have the following stronger version of the estimate~\eqref{e:NUP}:
\begin{equation}
  \label{e:NUP-reduced-2}
\|\indic_X\mathcal B\indic_Y\|_{L^2\to L^2}\leq Ch^{2\epsilon}.
\end{equation}
Here the constants $C,h_0$ depend only on $C_0,c_{\nlin}$, and some $(\zeta,\epsilon)$-dependent
$C^\infty$-seminorms of $\Theta,\chi_{\mathcal B}$.
\end{proposition}
\begin{proof}
1. We may assume that 
\begin{equation}
  \label{e:reverse-volume-XY}
|Y|\geq h^{5\epsilon} \mathfrak m_Y.
\end{equation}
Indeed, assume that instead $|Y|< h^{5\epsilon}\mathfrak m_Y$. Then 
we have by~\eqref{e:uncertainty-volume-2} and~\eqref{e:NUP-XY-vol}
$$
\|\indic_X\mathcal B\indic_Y\|_{L^2\to L^2}\leq Ch^{-\frac12}\sqrt{|X|\cdot |Y|}
\leq Ch^{2\epsilon}
$$
From the $r_1(h)$-uniformity of $X,Y$, \eqref{e:uniformity-volume},
and~\eqref{e:usd-Nr} we have for $h$ small enough and all $r\in [h,1]$
\begin{equation}
  \label{e:uniformity-volume-used}
|X|\geq h^{\epsilon/10}|X|_r\,\mathbf M_r(X),\quad
|Y|\geq h^{\epsilon/10}|Y|_r\,\mathbf M_r(Y).
\end{equation}
We may also assume that for all $r\in [h,1]$
\begin{equation}
  \label{e:lower-spread-bound}
\mathbf M_r(X)|Y|_{h/r}\geq h^{4\epsilon}r,\quad
\mathbf M_r(Y)|X|_{h/r}\geq h^{4\epsilon}r.
\end{equation}
Indeed, assume that instead $\mathbf M_r(X)|Y|_{h/r}<h^{4\epsilon}r$ for some $r\in [h,1]$.
Then by Lemma~\ref{l:L2-intermediate} and~\eqref{e:volume-nbhd} we have
$$
\begin{aligned}
\|\indic_X\mathcal B\indic_Y\|_{L^2\to L^2}
&\leq Ch^{-\frac12}\sqrt{\mathbf M_r(x)|Y(h/r)|}
\\
&\leq Cr^{-\frac12}\sqrt{\mathbf M_r(x)|Y|_{h/r}}
\leq Ch^{2\epsilon}
\end{aligned}
$$
which gives~\eqref{e:NUP-reduced-2}. We argue similarly if $\mathbf M_r(Y)|X|_{h/r}<h^{4\epsilon}r$.

Combining~\eqref{e:uniformity-volume-used} and~\eqref{e:NUP-XY-vol},
we get for all $r\in [h,1]$
$$
\mathbf M_r(X)|Y|_{h/r}\mathbf M_{h/r}(Y)|X|_r\leq h^{-\epsilon}|X|\cdot |Y|
\leq Ch^{1-2\epsilon}.
$$
Together with~\eqref{e:lower-spread-bound} (with $r$ replaced by~$h/r$) this gives a converse bound to it:
\begin{equation}
  \label{e:upper-spread-bound}
\mathbf M_r(X)|Y|_{h/r}\leq Ch^{-6\epsilon}r,\quad
\mathbf M_r(Y)|X|_{h/r}\leq Ch^{-6\epsilon}r\quad\text{for all }r\in [h,1].
\end{equation}
Together with~\eqref{e:volume-nbhd} this gives
\begin{equation}
  \label{e:upper-spread-bound-2}
\mathbf M_r(X)|Y(h/r)|\leq Ch^{1-6\epsilon},\quad
\mathbf M_r(Y)|X(h/r)|\leq Ch^{1-6\epsilon}\quad\text{for all }r\in [h,1].
\end{equation}
Now Lemma~\ref{l:Linfty-L4} with $N:=1$ gives
\begin{equation}
  \label{e:NUPred2-int-1}
\begin{aligned}
\|\indic_X\mathcal B\indic_Y\|_{L^2\to L^2}\leq &\, Ch^{-\frac12}(|X|\cdot|Y|)^{\frac38}
(\mathbf M_{h^{\frac12-\epsilon}}(Y)|X(h^{\frac12})|)^{\frac18}\cdot
\\
&\cdot \mathbf E_{h^{\frac12-\epsilon}}(X(h^{\frac12})\xrightarrow\Theta Y)^{\frac18}
+Ch
\\
\leq &\, Ch^{-\frac54\epsilon}\,\mathbf E_{h^{\frac12-\epsilon}}(X(h^{\frac12})\xrightarrow\Theta Y)^{\frac18}
+Ch
\end{aligned}
\end{equation}
where we bound $|X|\cdot|Y|$ by~\eqref{e:NUP-XY-vol} and
and $\mathbf M_{h^{\frac12-\epsilon}}(Y)|X(h^{\frac12})|$
by~\eqref{e:upper-spread-bound-2} (with $r:=h^{\frac12}$) and~\eqref{e:M-r-monotone}.

\noindent 2. We now verify that the sets $X,Y$ satisfy the conditions~\eqref{e:NUPa-Y-uni}--\eqref{e:NUPa-X-noncon} of Proposition~\ref{l:NUP-additive} for $h$ small enough:
\begin{itemize}
\item[\underline{\eqref{e:NUPa-Y-uni}:}] this follows from~\eqref{e:uniformity-volume-used} with $r:=h^{\frac12}$.
\smallskip

\item[\underline{\eqref{e:NUPa-Y-noncon1}:}] we have
$$
{\mathbf M_{|Y(h^{\frac12})|}(Y(h^{\frac12}))\over |Y(h^{\frac12})|}
\leq  {h^{-\epsilon/2}\,\mathbf M_{|Y(h^{\frac12})|}(Y)\over |Y|}
\leq  {h^{\zeta-\epsilon}\mathfrak m_Y\over |Y|}
\leq h^{\zeta-6\epsilon}\leq h^{\zeta/2}.
$$
Here in the first inequality we use the uniformity of~$Y$, \eqref{e:usd-Nr}, and Lemma~\ref{l:uniformity-volume-more}
with $\tilde r:=h^{\frac12}$, $r:=|Y(h^{\frac12})|$.
In the second inequality we use~\eqref{e:NUP-XY-1}.
In the third inequality we use~\eqref{e:reverse-volume-XY}.

\item[\underline{\eqref{e:NUPa-Y-noncon2}:}] we have
$$
\mathbf M_{h^\beta}(Y)\leq C_0h^{33\epsilon}\mathfrak m_Y\leq h^{27\epsilon}|Y|.
$$
Here in the first inequality we use~\eqref{e:NUP-XY-2}
and in the second inequality we use~\eqref{e:reverse-volume-XY}.
\item[\underline{\eqref{e:NUPa-X-noncon}:}] let $r\in [h^{\frac12},1]$. We estimate
$$
{\mathbf M_r(X(h^{\frac12}))\over |X(h^{\frac12})|}
\leq {h^{-\epsilon/2}\,\mathbf M_r(X)\over |X|}
\leq {h^{-\epsilon}\over |X|_r}
\leq h^{-1-5\epsilon}r\mathbf M_{h/r}(Y)
\leq h^{-\beta-6\epsilon}r^\zeta
\leq h^{-2\beta}r^\zeta.
$$
Here in the first inequality we use the uniformity of~$X$, \eqref{e:usd-Nr}, and Lemma~\ref{l:uniformity-volume-more}
with $\tilde r:=h^{\frac12}$. In the second inequality we use~\eqref{e:uniformity-volume-used}.
In the third inequality we use~\eqref{e:lower-spread-bound}.
In the fourth inequality we use~\eqref{e:NUP-XY-3} with $r$ replaced by $h/r$.
\end{itemize}
Now Proposition~\ref{l:NUP-additive} applies and gives the estimate~\eqref{e:NUP-additive-conc}:
$$
\mathbf E_{h^{\frac12-\epsilon}}(X(h^{\frac12})\xrightarrow\Theta Y)\leq h^{26\epsilon}
$$
which together with~\eqref{e:NUPred2-int-1} gives the needed bound~\eqref{e:NUP-reduced-2}.
\end{proof}
We finally reduce the general case to the uniform case using the decomposition
of Lemma~\ref{l:uniform-set-decomposition}:
\begin{proof}[Proof of Theorem~\ref{t:NUP}]
As explained in~\S\ref{s:NUP-basic-reductions}, we may assume that the conditions~\eqref{e:NUPred}--\eqref{e:XY-contained} hold.
Using Lemma~\ref{l:uniform-set-decomposition}, we write
$$
\begin{gathered}
X = X_{\uni}\sqcup X_{\res},\quad
Y = Y_{\uni}\sqcup Y_{\res},\quad
X_{\uni} = \bigsqcup_{j=1}^{N_X} X_j,\quad 
Y_{\uni} = \bigsqcup_{k=1}^{N_Y} Y_k,
\\
N_X,N_Y\leq N_1(h),\quad
|X_{\res}|\leq h|X|,\quad
|Y_{\res}|\leq h|Y|,
\end{gathered}
$$
where the sets $X_j,Y_k$ are $r_1(h)$-uniform on scales~$h$ to~1.

By the triangle inequality, we have
$$
\begin{aligned}
\|\indic_X\mathcal B\indic_Y\|_{L^2\to L^2}
&\leq \|\indic_{X_{\uni}}\mathcal B\indic_{Y_{\uni}}\|_{L^2\to L^2}
+\|\indic_{X_{\res}}\mathcal B\indic_Y\|_{L^2\to L^2}
+\|\indic_X\mathcal B\indic_{Y_{\res}}\|_{L^2\to L^2}
\\
&\leq N_1(h)^2\max_{j,k}\|\indic_{X_j}\mathcal B\indic_{Y_k}\|_{L^2\to L^2}
\\
&\quad +\|\indic_{X_{\res}}\mathcal B\indic_Y\|_{L^2\to L^2}
+\|\indic_X\mathcal B\indic_{Y_{\res}}\|_{L^2\to L^2}.
\end{aligned}
$$
We have by~\eqref{e:uncertainty-volume-2} and~\eqref{e:NUP-XY-vol}
$$
\|\indic_{X_{\res}}\mathcal B\indic_Y\|_{L^2\to L^2}
\leq Ch^{-\frac12}\sqrt{|X_{\res}|\cdot |Y|}\leq Ch^{1-\epsilon\over 2}
\leq Ch^{\epsilon}
$$
and $\|\indic_X\mathcal B\indic_{Y_{\res}}\|_{L^2\to L^2}$ is estimated similarly.

By Proposition~\ref{l:NUP-reduced-2} applied to the sets $X_j,Y_k$ (which are subsets of $X,Y$ and thus still satisfy the conditions~\eqref{e:NUP-XY-1}--\eqref{e:NUP-XY-vol} as explained in a remark following Theorem~\ref{t:NUP}), we have
$$
\max_{j,k}\|\indic_{X_j}\mathcal B\indic_{Y_k}\|_{L^2\to L^2}\leq Ch^{2\epsilon}.
$$
Since $N_1(h)=\mathcal O(h^{-\epsilon/2})$ by~\eqref{e:usd-Nr}, we now get $\|\indic_X\mathcal B\indic_Y\|_{L^2\to L^2}\leq Ch^\epsilon$ as needed.
\end{proof}

\subsection{Additive energy bound via projection theory}
\label{s:NUP-additive-proof}

In this section we prove Proposition~\ref{l:NUP-additive}, which finishes
the proof of Theorem~\ref{t:NUP}.

\subsubsection{A projection theorem}

Our proof relies on a result in projection theory,
Theorem~\ref{t:shmerkin-projection} below.
Before stating it, we need to introduce some notation. Assume that
$G:[0,1]^3\to \mathbb R$ is a $C^\infty$ function (which we call a \emph{projection function}).
We denote the argument of $G$ by $(x,\tilde y)$ where $x\in [0,1]$ and
$\tilde y=(y_1,y_2)\in [0,1]^2$, and make the assumption that
$\partial_{\tilde y}G(x,\tilde y)\neq 0$ on $[0,1]^3$.
Define the polar angle of the derivative $\partial_{\tilde y}G$ as follows:
$$
\theta(x,\tilde y):={\partial_{\tilde y}G(x,\tilde y)\over |\partial_{\tilde y}G(x,\tilde y)|}
\ \in\ \mathbb S^1.
$$
The projections we study are the maps $\tilde y\mapsto G(x,\tilde y)$, with $x$
treated as a parameter. The angle $\theta(x,\tilde y)$ determines the tangent
space of the fiber $\{\tilde y'\colon G(x,\tilde y')=G(x,\tilde y)\}$
at the point $\tilde y$. Here are a couple of examples of projection functions:
\begin{align}
  \label{e:projection-example-1}
G(x,\tilde y)&=y_1,\quad \theta(x,\tilde y)=(1,0);
\\
  \label{e:projection-example-2}
G(x,\tilde y)&=y_1+xy_2,\quad \theta(x,\tilde y)=(1+x^2)^{-\frac12}(1,x).
\end{align}
Note that the first example does not depend on $x$ while in the second example,
$\theta(x,\tilde y)$ varies with~$x$.

We will use the $d=2$ case of~\cite[Theorem~1.7]{Shmerkin-nonlinear};
the statement below is as written down in~\cite[Theorem~2.1]{Raz-Zahl},
with some changes in notation.
\begin{theorem}
  \label{t:shmerkin-projection}
Let $\zeta',C>0$. Then there exist $\beta'=\beta'(\zeta')>0$
and $\delta_0=\delta_0(\zeta',C)>0$ such that for each $\delta\in (0,\delta_0)$, the following holds.
Let $\widetilde Y_1\subset [0,1]^2$ be a union of $\delta$-cubes and $X_1\subset [0,1]$
be a union of $\delta$-intervals. Assume that:
\begin{itemize}
\item 
The projection function $G:[0,1]^3\to\mathbb R$ satisfies
\begin{equation}
  \label{e:shmerpro-1}
\sup_{x\in X_1}\|G(x,\bullet)\|_{C^2([0,1]^2)}<C,\quad
\inf_{(x,\tilde y)\in  X_1\times \widetilde Y_1}|\partial_{\tilde y}G(x,\tilde y)|>C^{-1}.
\end{equation}
\item The set $\widetilde Y_1$ satisfies a nonconcentration estimate
for every $|\widetilde Y_1|^{\frac12}$-cube $Q\subset [0,1]^2$:
\begin{equation}
  \label{e:shmerpro-2}
|\widetilde Y_1\cap Q|\leq \delta^{\zeta'}|\widetilde Y_1|.
\end{equation}
\item For each $\tilde y\in\widetilde Y_1$ and each arc $J_1\subset \mathbb S^1$
of length $|J_1|\geq \delta$, we have the transversality estimate
\begin{equation}
  \label{e:shmerpro-3}
\big|\{x\in X_1\colon \theta(x,\tilde y)\in J_1\}\big|\leq \delta^{-\beta'}|J_1|^{\zeta'}|X_1|.
\end{equation}
\end{itemize}
Then there exists a set $X_{\bad}\subset X_1$ with $|X_{\bad}|\leq \delta^{\beta'}|X_1|$ such that
for all $x\in X_1\setminus X_{\bad}$ and each set $\widetilde Y'\subset \widetilde Y_1$
with $|\widetilde Y'|\geq \delta^{\beta'}|\widetilde Y_1|$ we have
\begin{equation}
  \label{e:shmerpro-4}
|G(x,\widetilde Y')|_\delta\geq \delta^{-\beta'}|\widetilde Y_1|_\delta^{\frac12}.
\end{equation}
\end{theorem}
\begin{remark}
The transversality assumption~\eqref{e:shmerpro-3} is where the nonlinear phase condition~\eqref{e:Theta-c-nlin}
will be used. Note that~\eqref{e:shmerpro-3} is never satisfied for the projection function $G(x,\tilde y)=y_1$
but it is satisfied for the projection function $G(x,\tilde y)=y_1+xy_2$ if $X_1$ is sufficiently spread out
(e.g. $X_1=[0,1]$).
\end{remark}
The following is a version of Theorem~\ref{t:shmerkin-projection}
reformulated in a way which will be more convenient for us to use.
\begin{corollary}
  \label{c:shmercor}
Assume that $\zeta'\in (0,1],C'>0$. Then there exist $\beta''=\beta''(\zeta')\in (0,0.1]$
and $h_0=h_0(\zeta',C')>0$ such that for each $h\in (0,h_0]$, the following holds.
Let $X\subset [0.1,0.9]$ and $\widetilde Y\subset [0.1,0.9]^2$ be nonempty. Assume that:
\begin{itemize}
\item The projection function $G:[0,1]^3\to\mathbb R$ satisfies
\begin{equation}
  \label{e:shmercor-1}
\|G\|_{C^2([0,1]^3)}\leq C';\quad
\partial_{\tilde y}G(x,\tilde y)\in [1,2]^2\quad\text{for all }(x,\tilde y)\in [0,1]^3.
\end{equation}
\item The set $\widetilde Y$ satisfies the nonconcentration estimate
\begin{equation}
  \label{e:shmercor-2}
\mathbf M_{|\widetilde Y(h^{\frac12})|^{\frac12}}(\widetilde Y(h^{\frac12}))
\leq h^{\zeta'} |\widetilde Y(h^{\frac12})|.
\end{equation}
\item For each $\tilde y\in\widetilde Y(h^{\frac12})$ and each interval $J\subset \mathbb R$
of length $|J|\in [h^{\frac12},2]$, we have the transversality estimate
\begin{equation}
  \label{e:shmercor-3}
\big|\{x\in X(h^{\frac12})\colon \log\partial_{y_2}G(x,\tilde y)-\log\partial_{y_1}G(x,\tilde y)\in J\}\big|\leq \tfrac13 h^{-\beta''}|J|^{\zeta'}|X(h^{\frac12})|.
\end{equation}
\end{itemize}
Then there exists a set $X_{\bad}\subset X(h^{\frac12})$ with $|X_{\bad}|\leq h^{\beta''}|X(h^{\frac12})|$
such that for all $x\in X(h^{\frac12})\setminus X_{\bad}$ and each set $\widetilde Y'\subset \widetilde Y(h^{\frac12})$ with $|\widetilde Y'|\geq h^{\beta''}|\widetilde Y(h^{\frac12})|$ we have
\begin{equation}
  \label{e:shmercor-4}
|G(x,\widetilde Y')|_{h^{\frac12}}\geq h^{-\beta''}|\widetilde Y|_{h^{\frac12}}^{\frac12}.
\end{equation}
\end{corollary}
\begin{proof}
We put
\begin{equation}
  \label{e:NUP-choice-1}
\begin{gathered}
\beta''(\zeta')=\min(\tfrac12\beta'(\zeta'),0.1),\\
C:=\max(2C',1),\quad
h_0(\zeta',C'):=\min(\tfrac12\delta_0(\zeta',C)^2,0.001)
\end{gathered}
\end{equation}
where $\beta'(\zeta')$ and $\delta_0(\zeta',C)$ are given in Theorem~\ref{t:shmerkin-projection}. We apply
Theorem~\ref{t:shmerkin-projection} with
$$
\delta:=h^{\frac12},\quad
X_1:=X(h^{\frac12}),\quad
\widetilde Y_1:=\widetilde Y(h^{\frac12}).
$$
The assumptions~\eqref{e:shmerpro-1}--\eqref{e:shmerpro-2} follow immediately from~\eqref{e:shmercor-1}--\eqref{e:shmercor-2}.

We now verify the assumption~\eqref{e:shmerpro-3}.
For each $(x,\tilde y)\in [0,1]^3$,
the polar angle $\theta(x,\tilde y)$ lies in the cone $\{y_1,y_2> 0,\
\tfrac12\leq {y_1/y_2}\leq 2\}$. We thus identify $\theta(x,\tilde y)$
with a number in $[\arctan(1/2),\arctan 2]$ and note that
$$
\log \partial_{y_2}G-\log\partial_{y_1}G=\log\tan\theta.
$$
Let $J_1\subset \mathbb S^1$ be an arc
of length $|J_1|\geq h^{\frac12}$. We may assume that
$J_1\subset [\arctan(1/2),\arctan 2]$.
If $\theta(x,\tilde y)\in J_1$, then
$$
\log\partial_{y_2}G(x,\tilde y)-\log\partial_{y_1}G(x,\tilde y)\in
J:=\log\tan J_1.
$$
The derivative of the function $\theta\mapsto \log\tan \theta$
is equal to $(\cos\theta\sin\theta)^{-1}$ and thus
takes values in the interval $[2,2.5]$ on $J_1$. Thus $J$ is an interval of length
$h^{\frac12}\leq |J|\leq 3|J_1|$. Now~\eqref{e:shmercor-3} implies~\eqref{e:shmerpro-3}.

We have verified that all the assumptions of Theorem~\ref{t:shmerkin-projection} apply.
Now the estimate~\eqref{e:shmerpro-4} gives~\eqref{e:shmercor-4},
since $|\widetilde Y_1|_{h^{\frac12}}\geq |\widetilde Y|_{h^{\frac12}}$ by~\eqref{e:monotone-set}.
\end{proof}

\subsubsection{Our projection function}
\label{s:our-projection-function}

We write points $\vec y\in [0,1]^4$ as
$$
\vec y=(y_1,y_2,y'_1,y'_2)=(\tilde y,\tilde y'),\quad
\tilde y=(y_1,y_2),\quad
\tilde y'=(y'_1,y'_2).
$$
Let $\Theta$ be a phase function satisfying~\eqref{e:NUPred}--\eqref{e:Theta-c-nlin}.
Define the projection function $G:[0,1]^3\to \mathbb R$ by
\begin{equation}
  \label{e:G-def}
G(x,\tilde y):=\partial_x\Theta(x,y_1)+\partial_x\Theta(x,y_2),
\end{equation}
so that the derivative of the function $\vec\Theta$ from~\eqref{e:Theta-vec-def}
is given by
\begin{equation}
  \label{e:G-prop}
\partial_x\vec\Theta(x,\tilde y,\tilde y')=G(x,\tilde y)-G(x,\tilde y').
\end{equation}
From our assumption~\eqref{e:Theta-mix-12} we see that $G$ satisfies~\eqref{e:shmercor-1}, where the constant
$C'$ depends only on the $C^3$ norm of~$\Theta$.

The following lemma will help us verify the transversality estimate~\eqref{e:shmercor-3}.
It uses the nonlinear phase condition~\eqref{e:Theta-c-nlin}.
See Figure~\ref{f:levelsets}.
\begin{figure}
\includegraphics{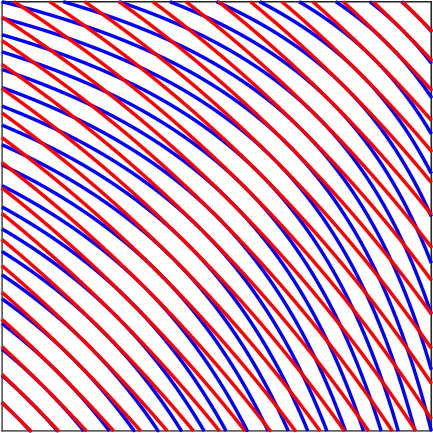}
\caption{An illustration of the use of the nonlinear phase assumption~\eqref{e:NUP-curved}.
Here we take $\Theta(x,y)=\log|x-y|$, $y\in [0,1]$, and plot the
level sets of the function $\tilde y\in [0,1]^2\mapsto G(x,\tilde y)$ defined in~\eqref{e:G-def}
for $x=2$ (in blue) and $x=5$ (in red). Note that the blue and red level sets
are transversal to each other outside of the diagonal $\{y_1=y_2\}$, which
is captured quantitatively in Lemma~\ref{l:transversality-basic}.}
\label{f:levelsets}
\end{figure}
%
\begin{lemma}
\label{l:transversality-basic}
Let $X\subset [0,1]$.
Assume that $\tilde y=(y_1,y_2)\in [0,1]^2$ and $|y_1-y_2|\geq r$ for some $r>0$.
Then for each interval $J\subset \mathbb R$, we have
\begin{equation}
  \label{e:transversality-basic}
\big|\{x\in X\colon \log\partial_{y_2}G(x,\tilde y)-\log\partial_{y_1}G(x,\tilde y)\in J\}\big|\leq \Big\lceil {1\over c_{\nlin}r}\Big\rceil\mathbf M_{|J|}(X).
\end{equation}
\end{lemma}
\begin{proof}
Define the function
$$
V(x):=\log\partial_{y_2}G(x,\tilde y)-\log\partial_{y_1}G(x,\tilde y),\quad
x\in [0,1].
$$
Then
$$
\begin{aligned}
V'(x)
&=\partial_x\log\partial_x\partial_y\Theta(x,y_2)
-\partial_x\log\partial_x\partial_y\Theta(x,y_1)
\\
&=\int_{y_1}^{y_2}\partial_x\partial_y\log\partial_x\partial_y\Theta(x,y)\,dy.
\end{aligned}
$$
By~\eqref{e:Theta-c-nlin}, we see that
$$
|V'(x)|\geq c_{\nlin}|y_1-y_2|\geq c_{\nlin}r\quad\text{for all }x\in [0,1].
$$
Therefore, the set $V^{-1}(J)$ is an interval of length
$\leq |J|/(c_{\nlin}r)$. It follows that
$$
|X\cap V^{-1}(J)|\leq \mathbf M_{|J|/(c_{\nlin}r)}(X)
$$
which together with~\eqref{e:M-r-monotone} gives~\eqref{e:transversality-basic}.
\end{proof}

\subsubsection{From large additive energy to small projection}

The next lemma shows that if the additive quantity in~\eqref{e:NUP-additive-conc} is large, then
the set $Y(h^{\frac12})^2$ has a large subset with a small projection,
for many values of~$x$.
\begin{lemma}
  \label{l:energy-reduced}
Let $\alpha_1,\tau_1,\epsilon\in (0,0.1]$. Then for all $h$ small enough depending
on $\alpha_1,\tau_1,\epsilon$ the following holds.
Let $G$ be given by~\eqref{e:G-def}
and $X,Y\subset[0.1,0.9]$ satisfy
\begin{equation}
  \label{e:energy-reduced-1}
\mathbf E_{h^{\frac12-\epsilon}}(X(h^{\frac12})\xrightarrow\Theta Y)\geq h^{\alpha_1}.
\end{equation}
Assume also that $Y$ satisfies a weak uniformity condition at scale $h^{\frac12}$
and a nonconcentration condition at scale $h^{\tau_1}$:
\begin{align}
  \label{e:Y-energy-uniform}
|Y|_{h^{\frac12}}\mathbf M_{h^{\frac12}}(Y)
&\leq h^{-\epsilon/2}|Y|,
\\
  \label{e:Y-energy-noncon}
\mathbf M_{h^{\tau_1}}(Y)
&\leq \tfrac1{32}h^{\alpha_1}|Y|.
\end{align}
Then we can choose for each $x\in X(h^{\frac12})$ a set
\begin{equation}
  \label{e:energy-reduced-Y}
\widetilde Y_x\subset \{(y_1,y_2)\in Y(h^{\frac12})^2\colon |y_1-y_2|\geq h^{\tau_1}\}
\end{equation}
such that
\begin{align}
  \label{e:energy-reduced-2}
\big|\{(x,\tilde y)\colon x\in X(h^{\frac12}),\ \tilde y\in \widetilde Y_x\}\big|
&\geq \tfrac1{144}h^{\alpha_1+\epsilon} |X(h^{\frac12})|\cdot |Y(h^{\frac12})|^2,
\\
  \label{e:energy-reduced-3}
\sup_{x\in X(h^{\frac12})}|G(x,\widetilde Y_x)|_{h^{\frac12}}&\leq 28h^{-\alpha_1-\epsilon}|Y|_{h^{\frac12}}.
\end{align}
\end{lemma}
\begin{proof}
1.
For $x\in [0,1]$, $z\in\mathbb R$, and $r>0$, define the nonnegative function
$$
W_r(x,z):=\big|
\{\tilde y\in Y^2\colon |G(x,\tilde y)-z|\leq r\}
\big|.
$$
By~\eqref{e:Theta-mix-12}, for each $x,y_1\in [0,1]$ and $z\in\mathbb R$ the set $\{y_2\in [0,1]\colon |G(x,y_1,y_2)-z|\leq r\}$
is contained in an interval of length $2r$.
Thus by~\eqref{e:M-r-monotone} we have the bound
\begin{equation}
  \label{e:W-r-bound}
W_r(x,z)\leq 2|Y|\,\mathbf M_r(Y).
\end{equation}
Recalling~\eqref{e:E-r-def} and~\eqref{e:G-prop}, we have
$$
\mathcal E_r=\{(x,\tilde y,\tilde y')\in [0,1]^5\colon
|G(x,\tilde y)-G(x,\tilde y')|\leq r
\}
$$
and thus
$$
\big|(X(h^{\frac12})\times Y^4)\cap \mathcal E_r\big|
=\int_{X(h^{\frac12})\times Y^2}W_r(x,G(x,\tilde y))\,d\tilde ydx.
$$
Therefore \eqref{e:energy-reduced-1} implies
\begin{equation}
  \label{e:enr-2}
\int_{X(h^{\frac12})\times Y^2}W_{h^{\frac12-\epsilon}}(x,G(x,\tilde y))\,d\tilde y dx\geq h^{\alpha_1}
|X(h^{\frac12})|\cdot |Y|^3\mathbf M_{h^{\frac12-\epsilon}}(Y).
\end{equation}

\noindent 2. For $x\in [0,1]$, define the sets
$$
\begin{aligned}
Z_x\,&:=\{z\in \mathbb R\colon W_{h^{\frac12-\epsilon}}(x,z)\geq \tfrac14 h^{\alpha_1}|Y|\,\mathbf M_{h^{\frac12-\epsilon}}(Y)\},
\\
\widehat Y_x\,&:=\{\tilde y=(y_1,y_2)\in Y^2\colon G(x,\tilde y)\in Z_x,\ |y_1-y_2|\geq 2 h^{\tau_1}\}.
\end{aligned}
$$
We split the domain of integration in~\eqref{e:enr-2} into three regions:
\begin{enumerate}
\item $G(x,\tilde y)\notin Z_x$. The integral over this region is bounded above by
$$
\tfrac14 h^{\alpha_1}|X(h^{\frac12})|\cdot |Y|^3\mathbf M_{h^{\frac12-\epsilon}}(Y).
$$
\item $G(x,\tilde y)\in Z_x$ and $|y_1-y_2|< 2h^{\tau_1}$. 
By~\eqref{e:M-r-monotone}  the volume of this region is bounded above by $4|X(h^{\frac12})|\cdot |Y|\,\mathbf M_{h^{\tau_1}}(Y)$.
By~\eqref{e:W-r-bound} and the nonconcentration condition~\eqref{e:Y-energy-noncon} the integral over this region is bounded above by
$$
8|X(h^{\frac12})|\cdot |Y|^2\,\mathbf M_{h^{\tau_1}}(Y)\mathbf M_{h^{\frac12-\epsilon}}(Y)
\leq \tfrac14 h^{\alpha_1}|X(h^{\frac12})|\cdot |Y|^3\mathbf M_{h^{\frac12-\epsilon}}(Y).
$$
\item $\tilde y\in \widehat Y_x$. By~\eqref{e:W-r-bound} the integral over this region is bounded above by
$$
2|Y|\,\mathbf M_{h^{\frac12-\epsilon}}(Y)
\big|\{(x,\tilde y)\in X(h^{\frac12})\times Y^2\colon \tilde y\in \widehat Y_x\}\big|.
$$
\end{enumerate}
Adding these together, we see that~\eqref{e:enr-2} implies that the sets $\widehat Y_x$ have large volume:
\begin{equation}
  \label{e:enr-2.5}
\big|\{(x,\tilde y)\in X(h^{\frac12})\times Y^2\colon \tilde y\in \widehat Y_x\}\big|
\geq \tfrac14 h^{\alpha_1}|X(h^{\frac12})|\cdot|Y|^2.
\end{equation}

\noindent 3. For each $x\in [0,1]$, define the set
$$
\widetilde Y_x:=\widehat Y_x(h^{\frac12})\ \subset\ Y(h^{\frac12})^2\cap \{|y_1-y_2|\geq h^{\tau_1}\}. 
$$
By~\eqref{e:volume-nbhd} we have
$$
|\widetilde Y_x|\geq\tfrac14 h|\widehat Y_x|_{h^{\frac12}},\quad
|Y(h^{\frac12})|\leq 3h^{\frac12}|Y|_{h^{\frac12}}. 
$$
By~\eqref{e:monotone-set}, we have for each $x\in [0,1]$
$$
\mathbf M_{h^{\frac12}}(\widehat Y_x)\leq \mathbf M_{h^{\frac12}}(Y^2)=\mathbf M_{h^{\frac12}}(Y)^2.
$$
Thus by \eqref{e:volume-covering-M} and the uniformity condition~\eqref{e:Y-energy-uniform}
$$
|\widehat Y_x|
\leq |\widehat Y_x|_{h^{\frac12}}\mathbf M_{h^{\frac12}}(Y)^2
\leq {h^{-\epsilon}|\widehat Y_x|_{h^{\frac12}}|Y|^2\over |Y|^2_{h^{\frac12}}}
\leq {36h^{-\epsilon}|\widetilde Y_x|\cdot |Y|^2\over |Y(h^{\frac12})|^2}.
$$
Now~\eqref{e:enr-2.5} implies~\eqref{e:energy-reduced-2}.

\noindent 4. It remains to show~\eqref{e:energy-reduced-3}.
Fix any $x\in [0,1]$.
Let $z_1,\dots,z_L\in Z_x$ be a maximal $2h^{\frac12-\epsilon}$-separated set.
For each $j$, we have by the definition of $Z_x$
$$
\big|\{\tilde y\in Y^2\colon |G(x,\tilde y)-z_j|<h^{\frac12-\epsilon}\}\big|
\geq \tfrac14 h^{\alpha_1}|Y|\,\mathbf M_{h^{\frac12-\epsilon}}(Y).
$$
The sets $\{\tilde y\in Y^2\colon |G(x,\tilde y)-z_j|<h^{\frac12-\epsilon}\}$
are nonintersecting for different $j$. Therefore
the sum of their volumes is bounded above by $|Y|^2$, which implies
\begin{equation}
  \label{e:enr-4}
L\leq {4h^{-\alpha_1}|Y|\over \mathbf M_{h^{\frac12-\epsilon}}(Y)}\leq 4h^{-\alpha_1}|Y|_{h^{\frac12-\epsilon}}\leq 4h^{-\alpha_1}|Y|_{h^{\frac12}}
\end{equation}
where in the second inequality we used~\eqref{e:volume-covering-M}
and in the third inequality we used~\eqref{e:r-cover-monotone}.

By the definitions of $\widehat Y_x,\widetilde Y_x$ and~\eqref{e:Theta-mix-12} we have
$$
G(x,\widetilde Y_x)\subset Z_x(4h^{\frac12}).
$$
We next have $Z_x\subset \bigcup_{j=1}^L [z_j-2h^{\frac12-\epsilon},z_j+2h^{\frac12-\epsilon}]$,
which implies that
$$
G(x,\widetilde Y_x)\subset \bigcup_{j=1}^L [z_j-3h^{\frac12-\epsilon},z_j+3h^{\frac12-\epsilon}].
$$
From here and~\eqref{e:r-cover-monotone} it follows that
$$
|G(x,\widetilde Y_x)|_{h^{\frac12}}\leq 7h^{-\epsilon}|G(x,\widetilde Y_x)|_{6h^{\frac12-\epsilon}}\leq 7h^{-\epsilon}L.
$$
Together with~\eqref{e:enr-4} this gives~\eqref{e:energy-reduced-3}.
\end{proof}

\subsubsection{Estimating the additive energies}

We are finally ready to give
\begin{proof}[Proof of Proposition~\ref{l:NUP-additive}]
1. Let $\zeta\in (0,0.1]$ and choose
\begin{equation}
  \label{e:NUP-choice-2}
\zeta':=\tfrac12\zeta,\quad
\beta'':=\beta''(\zeta'),\quad
\beta:=\min(\tfrac15\beta'',\tfrac12\zeta)
\end{equation}
where the function $\beta''(\zeta')$ is given in Corollary~\ref{c:shmercor}.
Assume that $\epsilon\in (0,0.1\beta]$, the phase function $\Theta$ satisfies~\eqref{e:NUPred}--\eqref{e:Theta-c-nlin}, $h>0$ is small, and the sets
$X,Y\subset [0.1,0.9]$ satisfy the conditions~\eqref{e:NUPa-Y-uni}--\eqref{e:NUPa-X-noncon}.

Define the set
$$
\widetilde Y:=\{(y_1,y_2)\in Y^2\colon |y_1-y_2|\geq \tfrac12 h^{\beta''/2} \}.
$$
Note that
\begin{equation}
  \label{e:NUPa-int-1}
\begin{gathered}
Y(h^{\frac12})^2\cap\{ |y_1-y_2|\geq h^{\beta''/2}\}
\ \subset\ \widetilde Y(h^{\frac12})
\ \subset\ Y(h^{\frac12})^2\cap \{|y_1-y_2|\geq \tfrac14h^{\beta''/2}\}.
\end{gathered}
\end{equation}
The set $\widetilde Y$ contains most of the volume of~$Y^2$. Indeed,
$$
|Y^2\setminus\widetilde Y|\leq |Y|\,\mathbf M_{h^{\beta''/2}}(Y)
\leq  |Y|\,\mathbf M_{h^{\beta}}(Y)
\leq h^{27\epsilon}|Y|^2 
$$
where the last inequality follows from~\eqref{e:NUPa-Y-noncon2}.
It follows that
\begin{equation}
  \label{e:NUPa-int-2}
|\widetilde Y|\geq \tfrac12|Y|^2.
\end{equation}
Using the weak uniformity condition~\eqref{e:NUPa-Y-uni}, we can convert~\eqref{e:NUPa-int-2}
to a lower bound on the volume of the neighborhood $\widetilde Y(h^{\frac12})$.
By~\eqref{e:volume-nbhd} we have
$$
|\widetilde Y(h^{\frac12})|\geq \tfrac14 h|\widetilde Y|_{h^{\frac12}},\quad
|Y(h^{\frac12})|\leq 3h^{\frac12}|Y|_{h^{\frac12}}.
$$
Next, by~\eqref{e:monotone-set} we have
$$
\mathbf M_{h^{\frac12}}(\widetilde Y)\leq \mathbf M_{h^{\frac12}}(Y^2)=\mathbf M_{h^{\frac12}}(Y)^2.
$$
Thus by~\eqref{e:volume-covering-M} and~\eqref{e:NUPa-Y-uni} we get
$$
|\widetilde Y|\leq |\widetilde Y|_{h^{\frac12}}\mathbf M_{h^{\frac12}}(Y)^2
\leq {h^{-\epsilon}|\widetilde Y|_{h^{\frac12}}|Y|^2\over |Y|_{h^{\frac12}}^2}
\leq {36h^{-\epsilon}|\widetilde Y(h^{\frac12})|\cdot |Y|^2\over |Y(h^{\frac12})|^2}.
$$
Together with~\eqref{e:NUPa-int-2} this gives
\begin{equation}
  \label{e:NUPa-int-3}
|\widetilde Y|_{h^{\frac12}}\geq \tfrac12h^{\epsilon}|Y|_{h^{\frac12}}^2,\quad
|\widetilde Y(h^{\frac12})|\geq \tfrac{1}{72}h^{\epsilon} |Y(h^{\frac12})|^2.
\end{equation}

\noindent 2. We apply Corollary~\ref{c:shmercor} to the sets $X$ and $\widetilde Y$,
with the projection function~$G$ defined in~\eqref{e:G-def}.
For that we need to verify the assumptions~\eqref{e:shmercor-1}--\eqref{e:shmercor-3}:
\begin{itemize}
\item[\underline{\eqref{e:shmercor-1}:}] see~\S\ref{s:our-projection-function} above.
\smallskip
\item[\underline{\eqref{e:shmercor-2}:}] we have
$$
\mathbf M_{|\widetilde Y(h^{\frac12})|^{\frac12}}(\widetilde Y(h^{\frac12}))
\leq \mathbf M_{|Y(h^{\frac12})|}(Y(h^{\frac12}))^2
\leq h^{\zeta}|Y(h^{\frac12})|^2
\leq 72h^{\zeta-\epsilon}|\widetilde Y(h^{\frac12})|
\leq h^{\zeta'}|\widetilde Y(h^{\frac12})|.
$$
Here the first inequality follows from the inclusion $\widetilde Y(h^{\frac12})\subset Y(h^{\frac12})^2$. The second inequality follows from~\eqref{e:NUPa-Y-noncon1}.
The third inequality follows from~\eqref{e:NUPa-int-3}.
\item[\underline{\eqref{e:shmercor-3}:}] let $\tilde y=(y_1,y_2)\in\widetilde Y(h^{\frac12})$ and $J$ be an interval of length $|J|\in [h^{\frac12},2]$.
We estimate
$$
\begin{aligned}
\big|\{x\in X(h^{\frac12})\colon \log \partial_{y_2}G(x,\tilde y)-\log\partial_{y_1}G(x,\tilde y)\in J\}\big|
&\leq \Big\lceil
{4h^{-\beta''/2}\over c_{\nlin}}
\Big\rceil \mathbf M_{|J|}(X(h^{\frac12}))
\\
\leq Ch^{-\beta''/2-2\beta}|J|^\zeta|X(h^{\frac12})|
&\leq \tfrac13 h^{-\beta''}|J|^{\zeta'}|X(h^{\frac12})|.
\end{aligned}
$$
Here the first inequality follows from Lemma~\ref{l:transversality-basic},
since $|y_1-y_2|\geq \tfrac14h^{\beta''/2}$ by~\eqref{e:NUPa-int-1}.
The second inequality follows from~\eqref{e:NUPa-X-noncon}.
\end{itemize}
Now Corollary~\ref{c:shmercor} applies and shows that there exists
a set $X_{\bad}\subset X(h^{\frac12})$ with $|X_{\bad}|\leq h^{\beta''}|X(h^{\frac12})|$
such that for all $x\in X(h^{\frac12})\setminus X_{\bad}$
and $\widetilde Y'\subset \widetilde Y(h^{\frac12})$
\begin{equation}
  \label{e:shmercor-applied}
|\widetilde Y'|\geq h^{\beta''}|\widetilde Y(h^{\frac12})|
\quad\Longrightarrow\quad
|G(x,\widetilde Y')|_{h^{\frac12}}\geq h^{-\beta''}|\widetilde Y|_{h^{\frac12}}^{\frac12}
\geq \tfrac12h^{-\beta''+\epsilon/2}|Y|_{h^{\frac12}}.
\end{equation}
Here the last inequality follows from~\eqref{e:NUPa-int-3}.

\noindent 3.
We argue by contradiction, assuming that~\eqref{e:NUP-additive-conc} does not hold:
$$
\mathbf E_{h^{\frac12-\epsilon}}(X(h^{\frac12})\xrightarrow\Theta Y)> h^{26\epsilon}.
$$
We apply Lemma~\ref{l:energy-reduced} with
$$
\alpha_1:=26\epsilon,\quad
\tau_1:=\tfrac12\beta''.
$$
Note that the assumption~\eqref{e:Y-energy-uniform} follows from~\eqref{e:NUPa-Y-uni}
and the assumption~\eqref{e:Y-energy-noncon} follows from~\eqref{e:NUPa-Y-noncon2}:
$$
\mathbf M_{h^{\tau_1}}(Y)\leq \mathbf M_{h^{\beta}}(Y)\leq h^{27\epsilon}|Y|
\leq\tfrac{1}{32}h^{\alpha_1}|Y|.
$$
We now see that for each $x\in X(h^{\frac12})$ there exists a set (where
we use~\eqref{e:NUPa-int-1})
$$
\widetilde Y_x\subset \widetilde Y(h^{\frac12})\subset Y(h^{\frac12})^2
$$
which satisfies the conditions~\eqref{e:energy-reduced-2}--\eqref{e:energy-reduced-3}.

By~\eqref{e:energy-reduced-3}
and since $27\epsilon<\beta''-\epsilon/2$ we have for all $x\in X(h^{\frac12})$
$$
|G(x,\widetilde Y_x)|_{h^{\frac12}}\leq 28h^{-27\epsilon}|Y|_{h^{\frac12}}
<\tfrac12h^{-\beta''+\epsilon/2}|Y|_{h^{\frac12}}.
$$
Applying~\eqref{e:shmercor-applied} to $\widetilde Y':=\widetilde Y_x$,
we see that
$$
x\in X(h^{\frac12})\setminus X_{\bad}\quad\Longrightarrow\quad
|\widetilde Y_x|<h^{\beta''}|\widetilde Y(h^{\frac12})|.
$$
It follows that
$$
\begin{aligned}
\big|\{ (x,\tilde y)\colon x\in X(h^{\frac12}),\ \tilde y\in \widetilde Y_x \}\big|
&\leq 
|X_{\bad}|\cdot |Y(h^{\frac12})|^2
+h^{\beta''}|X(h^{\frac12})|\cdot |Y(h^{\frac12})|^2
\\
&\leq 2h^{\beta''}|X(h^{\frac12})|\cdot |Y(h^{\frac12})|^2
\\
&<\tfrac{1}{144}h^{\alpha_1+\epsilon}|X(h^{\frac12})|\cdot |Y(h^{\frac12})|^2
\end{aligned}
$$
where in the last inequality we use that $\beta''>27\epsilon=\alpha_1+\epsilon$.
Together with~\eqref{e:energy-reduced-2} this gives a contradiction and finishes the proof.
\end{proof}
\begin{remark}
  \label{r:NUP-choice}
Combining~\eqref{e:NUP-choice-1} and~\eqref{e:NUP-choice-2}, we see that
the dependence of $\beta$ on $\zeta$ in Theorem~\ref{t:NUP} is given by
\begin{equation}
  \label{e:NUP-choice-3}
\beta(\zeta)=\min(0.1\beta'(\tfrac12\zeta),\tfrac12\zeta,0.02)
\end{equation}
where $\beta'(\bullet)$ is given by Theorem~\ref{t:shmerkin-projection}.
\end{remark}

\bibliographystyle{alpha}
\bibliography{General,QC,Dyatlov,Projection}

\end{document}